\documentclass[12pt]{amsart}
\usepackage{amssymb}
\usepackage{amsmath}
\usepackage{amsfonts}
\usepackage{mathrsfs}
\usepackage[all]{xy}

\newcommand{\query}[1]%
{\mbox{}\marginpar{\raggedright\hspace{0pt}{\small\em #1}}}%

\title[Polyhedral divisors and torus actions over arbitrary fields]{Polyhedral divisors and torus actions of
 complexity one over arbitrary fields}

\author{Kevin Langlois}
\address{Universit\'e
Grenoble I, Institut Fourier, UMR 5582 CNRS-UJF, BP 74,
38402 St.\ Martin
d'H\`eres c\'edex, France}
\email{kevin.langlois@ujf-grenoble.fr}
\begin{document}

\maketitle

\
{\footnotesize

\
}

\begin{abstract}
We show that the presentation of affine $\mathbb{T}$-varieties of complexity one 
in terms of polyhedral divisors holds over an arbitrary field. 
We also describe a class of multigraded algebras over Dedekind domains.
We study how the algebra associated to a polyhedral divisor changes when we extend the scalars. 
As another application, we provide a combinatorial description of 
affine $\mathbf{G}$-varieties of complexity one over a field, where $\mathbf{G}$ is a (not necessarily split) 
torus, by using elementary 
facts on Galois descent. This class of affine $\mathbf{G}$-varieties is described via a new combinatorial
object, which we call (Galois) invariant polyhedral divisor.
\end{abstract}
{\footnotesize \tableofcontents}

\theoremstyle{plain}
\newtheorem{theorem}{Theorem}[section]
\newtheorem{lemme}[theorem]{Lemma}
\newtheorem{proposition}[theorem]{Proposition}
\newtheorem{corollaire}[theorem]{Corollary}
\newtheorem*{theorem*}{Theorem}

\theoremstyle{definition}
\newtheorem{definition}[theorem]{Definition}
\newtheorem{rappel}[theorem]{}
\newtheorem{conjecture}[theorem]{Conjecture}
\newtheorem{exemple}[theorem]{Example}
\newtheorem{notation}[theorem]{Notation}

\theoremstyle{remark}
\newtheorem{remarque}[theorem]{Remark}
\newtheorem{note}[theorem]{Note}

\section*{Introduction}
In this paper, we are interested in a combinatorial description of multigraded normal
affine algebras of complexity one. From a geometrical viewpoint, these algebras are related 
to the classification of algebraic torus actions of complexity one on affine varieties.
Let $k$ be a field. Consider a split algebraic torus $\mathbb{T}$ over $k$. 
Recall that a $\mathbb{T}$-variety is a normal variety over $k$ endowed with an effective 
$\mathbb{T}$-action. There exist several combinatorial descriptions of $\mathbb{T}$-varieties
in term of the convex geometry. See [Dol75, Pin77, Dem88, FZ03] for the Dolgachev-Pinkham-Demazure (D.P.D.) 
presentation, [KKMS73, Tim97, Tim08] for toric case and complexity one case, and [AH06, AHS08, AOPSV12]
for higher complexity. Most classical works on $\mathbb{T}$-varieties require the ground field $k$ to be 
algebraically closed of characteristic zero. It is worthwhile mentioning that 
the description of affine $\mathbb{G}_{m}$-varieties [Dem88] due to Demazure holds over any field.

Let us list the most important results of the paper.

- The Altmann-Hausen presentation of affine $\mathbb{T}$-varieties
of complexity one in terms of polyhedral divisor holds over an arbitrary field,
see Theorem $4.3$. 

- This description holds as well for an important class of multigraded algebras over Dedekind domains, see Theorem $2.5$.

- We study how the algebra associated to a polyhedral divisor changes when we extend the scalars, 
see $2.12$ and $3.9$.

- As another application, we provide a combinatorial description of 
affine $\mathbf{G}$-varieties of complexity one, where $\mathbf{G}$ is a (not necessarily split) 
torus over $k$, by using elementary 
facts on Galois descent. This class of affine $\mathbf{G}$-varieties is classified via a new combinatorial
object, which we call a (Galois) invariant polyhedral divisor, see Theorem $5.10$. 

Let us discuss these results in more details. We start with a simple case of varieties with an 
action of a split torus.
Recall that a split algebraic torus $\mathbb{T}$ of dimension $n$ defined over $k$
is an algebraic group over $k$ isomorphic to $\mathbb{G}_{m}^{n}$, where $\mathbb{G}_{m} = 
\mathbb{G}_{m,k}$ is the multiplicative algebraic group ${\rm Spec}\,k[t,t^{-1}]$.
Let $M = \rm Hom(\it\mathbb{T},\mathbb{G}_{\it m})$ be the
character lattice of the torus $\mathbb{T}$. Then defining a $\mathbb{T}$-action on an affine variety
$X$ is equivalent to fixing an $M$-grading on the algebra $A = k[X]$, where $k[X]$
is the coordinate ring of $X$. Following the classification of affine $\mathbb{G}_{m}$-surfaces [FK91],
we say as in [Lie13, $1.1$] that the $M$-graded algebra $A$ is \em elliptic \rm if the graded piece $A_{0}$ is reduced to $k$. 

Multigraded affine algebras are classified via a numerical invariant called
complexity. This invariant was  introduced in [LV83] for the classification of homogeneous spaces
under the action of a connected reductive group. Consider the field $k(X)$ of rational functions on $X$ and 
its subfield $K_{0}$ of $\mathbb{T}$-invariant functions. The complexity of the $\mathbb{T}$-action on $X$ is the 
transcendence degree of $K_{0}$ over the field $k$. 
Note that for the situation where $k$ is algebraically closed, 
the complexity is also the codimension of 
the general $\mathbb{T}$-orbit in $X$ (see [Ros63]). 

In order to describe affine $\mathbb{T}$-varieties of complexity one, 
we have to consider combinatorial 
objects coming from convex geometry and from the geometry of algebraic curves. Let $C$ be a regular curve 
over $k$. A point of $C$ is assumed to be a closed 
point, and in particular, not necessarily rational. Thus, the residue field extension of $k$ 
at any point of $C$
has finite degree. 

To reformulate our first result, we recall some notation introduced in [AH06, Section $1$].
Denote by $N = \rm Hom(\it\mathbb{G}_{\it m},
\mathbb{T})$ the lattice of one-parameter subgroups of the torus $\mathbb{T}$ which is the dual of the lattice $M$. 
 Let $M_{\mathbb{Q}} = \mathbb{Q}\otimes_{\mathbb{Z}}M$, 
$N_{\mathbb{Q}} = \mathbb{Q}\otimes_{\mathbb{Z}}N$ be the associated dual $\mathbb{Q}$-vector spaces
of $M,N$, respectively, and let $\sigma\subset N_{\mathbb{Q}}$ be a strongly convex polyhedral cone. We can define 
as in [AH06] a Weil divisor $\mathfrak{D} = \sum_{z\in C}\Delta_{z}\cdot z$ with $\sigma$-polyhedral 
coefficients in $N_{\mathbb{Q}}$, called polyhedral divisor of Altmann-Hausen. More precisely,
each $\Delta_{z}\subset N_{\mathbb{Q}}$ is a polyhedron with a tail cone $\sigma$ (see $2.1$) and 
$\Delta_{z} = \sigma$ for all but finitely many points $z\in C$. Denoting by $\kappa_{z}$ the residue 
field of the point $z\in C$ and by $[\kappa_{z}:k]\cdot \Delta_{z}$ the
image of $\Delta_{z}$ under the homothety of ratio $[\kappa_{z}:k]$,
the sum
\begin{eqnarray*}
\rm deg\,\it\mathfrak{D} = \sum_{z\in C}[\kappa_{z}:k]\cdot\rm\Delta_{\it z} 
\end{eqnarray*}
is a polyhedron in $N_{\mathbb{Q}}$. This sum may be seen as the finite Minkowski sum of all polyhedra 
$[\kappa_{z}:k]\cdot\Delta_{z}$ 
different from $\sigma$. Considering the dual cone $\sigma^{\vee}\subset M_{\mathbb{Q}}$ of $\sigma$,
we define an evaluation function  
\begin{eqnarray*}
\sigma^{\vee}\rightarrow \rm Div_{\it\mathbb{Q}}(\it C),\,\,\,
m\mapsto \mathfrak{D}(m) = \sum_{z\in C}\min_{l\in\rm\Delta_{\it z}\it}\,\langle m, l\it\rangle\cdot z 
\end{eqnarray*}
with values in the vector space $\rm Div_{\it\mathbb{Q}}(\it C)$ of Weil
$\mathbb{Q}$-divisors over $C$. As in the classical case [AH06, $2.12$] we introduce the technical
condition of properness for the polyhedral divisor $\mathfrak{D}$ (see $2.2$, $3.4$, $4.2$) that we recall
thereafter. 

\begin{definition}
A $\sigma$-polyhedral divisor $\mathfrak{D} = \sum_{z\in C}\Delta_{z}\cdot z$ 
is called \em proper \rm if it
satisfies one of the following conditions.
\begin{enumerate}
\item[\rm (i)]
$C$ is affine. 
\item[\rm (ii)]
$C$ is projective and $\rm deg\it\, \mathfrak{D}$
is strictly contained in the cone $\sigma$. Furthermore, 
if $\rm deg\it\, \mathfrak{D}(m) \rm = 0$, then $m$ belongs 
to the boundary of $\sigma^{\vee}$ and some non-zero 
 integral multiple of $\mathfrak{D}(m)$ is principal. 
\end{enumerate}
\end{definition}
For instance, if $C = \mathbb{P}^{1}_{k}$ is the projective line, then the polyhedral divisor 
$\mathfrak{D}$ is proper if and only if $\rm deg\,\it \mathfrak{D}$ is strictly included in $\sigma$.

For every affine variety $X$ with an effective $\mathbb{T}$-action,
we will call \emph{multiplicative system} of $k(X)$ a sequence $(\chi^{m})_{m\in M}$,
where each $\chi^{m}$ is a homogeneous element of $k(X)$ of degree $m$ satisfying the conditions
$\chi^{m}\cdot \chi^{m'} = \chi^{m + m'}$ for all $m,m'\in M$, and $\chi^{0} = 1$.
One of the main results of this paper can be stated as follows.
\begin{theorem}
Let $k$ be a field.
\begin{enumerate}
\item[\rm (i)] If $\mathfrak{D}$ is a proper $\sigma$-polyhedral 
divisor on a regular curve $C$ over $k$, then the $M$-graded algebra $A[C,\mathfrak{D}] = \bigoplus_{m\in\sigma^{\vee}\cap M}A_{m},$
where $$A_{m} = H^{0}(C,\mathcal{O}_{C}(\lfloor \mathfrak{D}(m)\rfloor)),$$
is the coordinate ring of an affine $\mathbb{T}$-variety of complexity one over $k$.
\item[\rm (ii)] Conversely, to any affine $\mathbb{T}$-variety $X = \rm Spec\,\it A$ of complexity one over $k$,
one can associate a pair $(C_{X}, \mathfrak{D}_{X,\gamma})$ as follows.
\begin{enumerate}
\item[(a)] $C_{X}$ is the abstract regular curve over $k$ defined by the conditions 
$k[C_{X}] = k[X]^{\mathbb{T}}$ and $k(C_{X}) = k(X)^{\mathbb{T}}$. 
\item[(b)] $\mathfrak{D}_{X,\gamma}$ is a proper $\sigma_{X}$-polyhedral divisor 
over $C_{X}$, which is uniquely determined by $X$ and by a multiplicative system $\gamma = (\chi^{m})_{m\in M}$ of $k(X)$.
\end{enumerate}
We have a natural identification $A = A[C_{X},\mathfrak{D}_{X,\gamma}]$ of $M$-graded algebras
with the property that every homogeneous element $f\in A$ of degree $m$ is equal to $f_{m}\chi^{m}$, 
for a unique global section $f_{m}$ of the sheaf $\mathcal{O}_{C_{X}}(\lfloor \mathfrak{D}_{X,\gamma}(m)\rfloor)$.
\end{enumerate} 
\end{theorem}
In the proof of assertion $\rm (ii)$, we use an effective calculation from [Lan13]. 
We divide the proof into two cases. In the \em non-elliptic case \rm we show that the assertion
holds more generally in the context of Dedekind domains. More precisely, we give a perfect dictionary
similar to $0.2 \rm (i),(ii)$ for $M$-graded algebras defined by a polyhedral divisor over
a Dedekind ring (see $2.2, 2.3$ and Theorem $2.5$). We deal in $2.6$ with an example 
of a polyhedral divisor over $\mathbb{Z}[\sqrt{-5}]$. In the \em elliptic case, \rm 
we consider an elliptic $M$-graded algebra $A$ over $k$ satisfying the 
assumptions of $0.2\,\rm (ii)$. By a well-known result (see [EGA II, $7.4$]),
we can construct a regular projective curve arising from the 
algebraic function field $K_{0} = (\rm Frac\,\it A)^{\mathbb{T}}$. In this construction, 
the points of $C$ are identified with the places of $K_{0}$. 
Then we show that the $M$-graded algebra is described by a polyhedral divisor 
over $C$ (see Theorem $3.5$).

Let us pass further to the general case of varieties with an action of a non necessarily split
torus. The reader may consult [Bry79, CTHS05, Vos82, ELST12]
for the theory of non-split toric varieties and [Hur11] for the spherical embeddings.
Let $\mathbf{G}$ be a torus over $k$; then $\mathbf{G}$ splits in a finite Galois
extension $E/k$. Let $\rm Var_{\it \mathbf{G}, E}(\it k)$ be the category
of affine $\mathbf{G}$-varieties of complexity one splitting in $E/k$ (see $5.4$).
For an object $X\in\rm Var_{\it \mathbf{G}, E}(\it k)$ let
$[X]$ be the isomorphism class and $X_{E} = X\times_{\rm Spec\,\it k}\rm Spec\,\it E$
be its extension of $X$ over the field extension. Fixing $X\in\rm Var_{\it \mathbf{G}, E}(\it k)$, 
as an application of our previous results, we study the pointed set
\begin{eqnarray*}
\left(\left\{[Y]\,|\, Y\in {\rm Var}_{\mathbf{G}, E}(k) \,\,\, {\rm and}\,\,\,X_{E}
\simeq_{{\rm Var}_{\mathbf{G}, E}(E)}
Y_{E}\right\},[X]\right) 
\end{eqnarray*}
of isomorphism classes of $E/k$-forms of $X$ that is in bijection with the first pointed 
set $H^{1}(E/k, \rm Aut_{\it\mathbf{G}_{E}}(\it X_{E}\rm ))$ of non abelian Galois cohomology.
By elementary arguments (see $5.7$) the latter pointed sets are described by all possible
homogeneous semi-linear $\mathfrak{G}_{E/k}$-actions on the multigraded algebra $E[X_{E}]$, 
where $\mathfrak{G}_{E/k}$ is the Galois group of $E/k$.
Translating this to the language of polyhedral divisors, we obtain a combinatorial description 
of $E/k$-forms of $X$, see Theorem $5.10$. This theorem can be viewed as a first step towards the study of the forms of 
$\mathbf{G}$-varieties of complexity one.

Let us give a brief summary of the contents of each section. In the first section, we recall how to extend the D.P.D.
presentation of parabolic graded algebra to the context of Dedekind domain.
 This fact has been mentioned in [FZ03] and 
firstly treated by Nagat Karroum
in a master dissertation [Kar04]. 
In the second and the third sections, 
we study respectively a class of multigraded algebras over Dedekind domains and a class of elliptic
multigraded algebras over a field. In the fourth section, we classify 
affine $\mathbb{T}$-varieties of complexity one. 
The last section is devoted to the non-split case.

\begin{rappel}
All considered rings are assumed to be commutative and unitary.
Let $k$ be a field. Given a lattice $M$ we let $k[M]$ be
the semigroup algebra 
\begin{eqnarray*}
\bigoplus_{m\in M}k\chi^{m},\,\,\,{\rm where}\,\,\,\chi^{m+m'}=\chi^{m}\cdot \chi^{m'}\,\,\,{\rm and}\,\,\,\chi^{0} = 1.
\end{eqnarray*}
Recall that a \em $\mathbb{Q}$-divisor \rm on a scheme $Y$ is a Weil divisor on $Y$ 
with rational coefficients. 
By a \em variety \rm $X$ over $k$ we mean an integral
separated scheme of finite type over $k$; one assumes in addition that $k$ is 
algebraically closed in the field of rational functions $k(X)$. In particular, $X$ is geometrically
irreducible.
\end{rappel}
\em  Acknowledgments. \rm The author is grateful to Karol Palka and Mikhail Zaidenberg for his remarks which helped 
to improve the text. We would like to thank Matthieu Romagny for kindly 
answering our questions, and Hanspeter Kraft for proposing to treat the non-split case.
We also thank the jury members of the Ph.D. thesis and the referee for many suggestions and corrections.
\section{Graded algebras over Dedekind domains}
In this section, we recall how to generalize 
the Dolgachev-Pinkham-Demazure (D.P.D.)
presentation in [FZ03, Section $3$] to the context
of Dedekind domains (see Lemma $1.6$). This generalization concerns in particular 
an algebraic description of affine normal parabolic complex $\mathbb{C}^{\star}$-surfaces. 
Let us recall the definition of a Dedekind domain.  
\begin{rappel}
An integral domain $A_{0}$ is called a \em Dedekind 
domain \rm (or a Dedekind ring) 
if it is not a field 
and if it satisfies the following conditions.
\begin{enumerate}
\item[(i)] The ring $A_{0}$ is noetherian.

\item[(ii)]The ring $A_{0}$ is integrally closed in 
its field of fractions.

\item[(iii)]Every nonzero prime ideal is a 
maximal ideal.
\end{enumerate}
\end{rappel}
Let us mention several classical examples 
of Dedekind domains.
\begin{exemple}
Let $K$ be a number field.
Then the ring $\mathbb{Z}_{K}$ of integers of $K$ 
is a Dedekind ring. 

Let $A$ be a finitely generated normal algebra of dimension one
over a field $k$. This means that the
scheme $C = \rm Spec\,\it A$ is a regular affine curve.
Then the coordinate ring $A = k[C]$ is Dedekind.

The algebra of power series $k[[t]]$ in one variable over the field $k$ is a Dedekind
domain. More generally, every principal ideal domain (and so every discrete valuation ring) that
is not a field is a Dedekind domain. 
\end{exemple}  
\begin{rappel}
Let $A_{0}$ be an integral domain, and let $K_{0}$ be 
its field of fractions. 
Recall that a \em fractional ideal \rm $\mathfrak{b}$ is a finitely
generated nonzero $A_{0}$-submodule of $K_{0}$. Actually, every
fractional ideal is of the form $\frac{1}{f}\cdot\mathfrak{a}$, 
where $f\in A_{0}$ is nonzero and 
$\mathfrak{a}$ is a nonzero ideal of $A_{0}$. If $\mathfrak{b}$ is
equal to $u\cdot A_{0}$ for some nonzero element $u\in K_{0}$, then we say that 
$\mathfrak{b}$ is a \em principal \rm fractional ideal.
\end{rappel}
The following result
gives a description of fractional ideals of $A_{0}$ in 
terms of Weil divisors
 on $Y = \rm Spec\,\it A_{\rm 0}$, where $A_{0}$ 
is a Dedekind domain.
This assertion is well known, and so the proof is omitted. 
\begin{lemme}
Let $A_{0}$ be a Dedekind ring with field of fractions $K_{0}$. Let
$Y =\rm Spec\,\it A_{\rm 0}$. Then the map
\begin{eqnarray*}
\rm Div_{\it \mathbb{Z}}(\it Y\rm )\rightarrow\rm Id(\it A_{\rm 0}\rm )\it , \,\,
D\mapsto H^{\rm 0}(Y,\mathcal{O}_{Y}( D))
\end{eqnarray*}
is a bijection between the set of integral Weil divisors on $Y$ and the set
of fractional ideals of $A_{0}$. 
Every fractional ideal is locally free of rank $1$ as $A_{0}$-module
and the natural map
\begin{eqnarray*}
H^{\rm 0}(Y,\mathcal{O}_{Y}( D))\otimes H^{\rm 0}(Y,\mathcal{O}_{Y}( D'))\rightarrow 
H^{\rm 0}(Y,\mathcal{O}_{Y}( D+D'))
\end{eqnarray*}
is surjective.
A Weil divisor $D$ on $Y$ is principal {\rm (}resp. effective{\rm )} if and only 
if the corresponding fractional ideal is principal {\rm (}resp. contains $A_{0}${\rm )}.
\end{lemme}
\begin{notation}
Let $A_{0}$ be a Dedekind domain. 
For a $\mathbb{Q}$-divisor $D$ 
on the affine scheme $Y = \rm Spec\,\it A_{\rm 0}$ we denote by $A_{0}[D]$ 
the graded algebra
\begin{eqnarray*}
\bigoplus_{i\in\mathbb{N}}
H^{0}(Y,\mathcal{O}_{Y}(\lfloor iD\rfloor ))\,t^{i},  
\end{eqnarray*}
where $t$ is a variable over the field 
$K_{0}$. Note that
$A_{0}[D]$ is normal as intersection of discrete 
valuation rings with field of fractions $K_{0}(t)$
(see the argument for [Dem88, $2.7$]).
\end{notation}
The next lemma provides
a D.P.D. presentation for a class of
graded subrings of $K_{0}[t]$. It will be
used in the next section. Here we give an elementary proof
using the description in $1.4$ of fractional ideals.
\begin{lemme}
Let $A_{0}$ be a Dedekind ring with the field of fractions $K_{0}$.
Let
\begin{eqnarray*}
A = \bigoplus_{i\in\mathbb{N}}A_{i}\,t^{i}\subset K_{0}[t]
\end{eqnarray*}
be a normal graded subalgebra of finite type over $A_{0}$, where every $A_{i}$
is contained in $K_{0}$. 
Assume that the field of fractions of $A$ is $K_{0}(t)$. 
Then there exists a unique $\mathbb{Q}$-divisor $D$ on 
$Y = \rm Spec\it\,A_{\rm 0}$ such that $A = A_{0}[D]$.
Furthermore we have $Y = \rm Proj\,\it A$.
\end{lemme} 
\begin{proof}
Theorem $1.4$ and Lemma $2.2$ in [GY83]
imply that every nonzero module $A_{i}$ can be written as  
\begin{eqnarray*}
A_{i} = H^{0}(Y,\mathcal{O}_{Y}(D_{i}))
\end{eqnarray*}
for some $D_{i}\in \rm Div_{\it \mathbb{Z}}(\it Y\rm )$.
By Proposition $3$ in [Bou72, III.$3$] 
there exists a positive integer $d$ such 
that the subalgebra
\begin{eqnarray*}
A^{(d)} := \bigoplus_{i\geq 0}A_{di}\,t^{di}
\end{eqnarray*} 
is generated by $A_{d}\,t^{d}$. 
Proceeding by induction, for any $i\in\mathbb{N}$
we have $D_{di} = iD_{d}$. Let $D = D_{d}/d$.
Then using the normality of $A$ and $A_{0}[D]$, we obtain 
for any homogenous element $f\in K_{0}[t]$ the
equivalences
\begin{eqnarray*}
f\in A_{0}[D]\Leftrightarrow f^{d}\in A_{0}[D]  
\Leftrightarrow f^{d}\in A \Leftrightarrow f\in A.
\end{eqnarray*}
This yields $A = A_{0}[D]$.

Let $D'$ be another $\mathbb{Q}$-divisor on $Y$ such that
$A = A_{0}[D']$. Comparing the graded pieces of
$A_{0}[D]$ and of $A_{0}[D']$, it follows that 
$\lfloor iD\rfloor = \lfloor iD'\rfloor$ for any $i\in\mathbb{N}$.
Hence $D = D'$ and so the decomposition is unique.

It remains to show the equality $Y = \rm Proj\,\it A$.
Let $V =  \rm Proj\,\it A$. By Exercice $5.13$ in [Har77, II]
and Proposition $3$ in [Bou72, III.$1$],
we may assume that $A = A_{0}[D]$ is generated by
$A_{1}t$. Since the sheaf $\mathcal{O}_{Y}(D)$ is 
locally free of rank one over $\mathcal{O}_{Y}$, there
exist $g_{1},\ldots,g_{s}\in A_{0}$ such that
\begin{eqnarray*}
Y = \bigcup_{j = 1}^{s}Y_{g_{j}},\,\,\,\rm  where\,\,\, 
\it Y_{g_{j}} = \rm Spec\,\it (A_{\rm 0\it})_{g_{j}}
\end{eqnarray*}
and such that for $e = 1,\ldots ,s$,
\begin{eqnarray*}
A_{1}\otimes_{A_{0}} (A_{0})_{g_{e}} =
 \mathcal{O}_{Y}(D)(Y_{g_{e}}) = h_{e}\cdot A_{0} 
\end{eqnarray*}
for some $h_{e}\in K_{0}^{\star}$. Let $\pi:V \rightarrow Y$ be
the natural morphism induced by the inclusion $A_{0}\subset A$.
The preimage of the open subset $Y_{g_{e}}$ under 
$\pi$ is
\begin{eqnarray*}
\rm Proj \it\, A\otimes_{A_{\rm 0\it}} (A_{\rm 0\it})_{g_{e}} =
\rm Proj \it\, (A_{\rm 0\it })_{g_{e}}
[A_{\rm 1 \it}\otimes_{A_{\rm 0\it}} (A_{\rm 0\it})_{g_{e}}t] =
\rm Proj \it\, (A_{\rm 0\it})_{g_{e}}[h_{e}t] = Y_{g_{e}}.
\end{eqnarray*}
Hence $\pi$ is the identity map and so $Y =V$, as required.  
\end{proof}
As an immediate consequence we obtain
the following. The reader can see that the proof of [FZ03, $3.9$] is applicable word by word to
positively
graded $2$-dimensional normal algebras of finite type over 
a Dedekind domain. 
\begin{lemme}
Let $A_{0}$ be a Dedekind ring with the field of fractions 
$K_{0}$, and let $t$ be a variable over $K_{0}$. Consider
the subalgebra
\begin{eqnarray*}
A = A_{0}[f_{1}t^{m_{1}},\ldots, f_{r}t^{m_{r}}]\subset K_{0}[t], 
\end{eqnarray*}
where $m_{1},\ldots ,m_{r}$ are positive integers and
$f_{1},\ldots , f_{r}\in K_{0}^{\star}$ are such that 
the field of fractions of $A$ is $K_{0}(t)$. 
Then the normalization of $A$ is equal to $A_{0}[D]$, where 
$D$ is the $\mathbb{Q}$-divisor\footnote{Let $D_{1},\ldots, D_{r}$
be $\mathbb{Q}$-divisors on a scheme $Y$. We define the minimum of $D_{1},\ldots, D_{r}$ by
letting
$$\min_{1\leq 1\leq r}\, D_{i} = \sum_{H\subset Y}\min_{1\leq 1\leq r}\{a_{i, H}\}\cdot H,$$
where for $i = 1,\ldots,r$, the number $a_{i,H}$ is the coefficient of $D_{i}$ corresponding
to the prime divisor $H\subset Y$.} 
\begin{eqnarray*}
 D = -\min_{1\leq i\leq r}\frac{\rm div\,\it f_{i}}{m_{i}}\,.
\end{eqnarray*}
\end{lemme}

\section{Multigraded algebras over Dedekind domains}

Let $A_{0}$ be a Dedekind ring and let $K_{0}$ be
its field of fractions. Given a lattice $M$, 
the purpose of this section 
is to study normal noetherian $M$-graded 
$A_{0}$-subalgebras of $K_{0}[M]$. 
We show below that these subalgebras
admit a description in terms of polyhedral divisors.
We start by recalling some necessary notation from 
convex geometry, see [AH06, Section $1$].
\begin{rappel}
Let $N$ be a lattice, and 
let $M = \rm Hom(\it N,\mathbb{Z}\rm )$ 
be its dual lattice. Denote by $N_{\mathbb{Q}} 
= \mathbb{Q}\otimes_{\mathbb{Z}}N$
and $M_{\mathbb{Q}} = \mathbb{Q}\otimes_{\mathbb{Z}}M$ the 
associated dual $\mathbb{Q}$-linear spaces, respectively. 
For any linear form $m\in M_{\mathbb{Q}}$
and for any vector $v\in N_{\mathbb{Q}}$ set $\left\langle m , v \right\rangle = m(v)$. A polyhedral cone $\sigma\subset N_{\mathbb{Q}}$
is called \em strongly convex \rm if it admits a vertex. 
This is equivalent to saying that the dual cone 
\begin{eqnarray*}
\sigma^{\vee} = \left\{\, m\in M_{\mathbb{Q}}\, | \,\forall v\in\sigma,\, 
\left\langle  m,v \right\rangle \geq 0\,\right\}
\end{eqnarray*}
is of full dimension.

Recall that for a nonzero strongly convex polyhedral cone $\sigma\subset N_{\mathbb{Q}}$
the \em Hilbert basis \rm $\mathscr{H}_{\sigma} = \mathscr{H}_{\sigma, N}$ of $\sigma$ in the 
lattice $N$ is the subset of all irreducible elements
\begin{eqnarray*}
\left\{v\in\sigma_{N}-\{0\}\,|\,\forall v_{1},v_{2}\in\sigma_{N}-\{0\}, v = v_{1}+v_{2}\Rightarrow v = v_{1}
\,\,\,\rm or\,\,\,\it v = v_{\rm 2}\it\right\}\,. 
\end{eqnarray*}
It is known that the set $\mathscr{H}_{\sigma}$ is finite and generates
the semigroup $(\sigma\cap N, +)$. Furthermore, it is minimal for
these latter properties. The cone $\sigma$ is said \em regular \rm
if $\mathscr{H}_{\sigma}$ is contained in a basis of $N$.

Let us fix a strongly convex polyhedral 
cone $\sigma\subset N_{\mathbb{Q}}$.
A subset $Q\subset N_{\mathbb{Q}}$ is a \em polytope \rm
if $Q$ is the convex hull 
of a non-empty finite subset of vectors.
We define $\rm Pol_{\it\sigma\rm}(\it N_{\mathbb{Q}}\rm)$
to be the set of polyhedra which can be written as the 
Minkowski sum $P = Q +\sigma$ with $Q$ a polytope 
of $N_{\mathbb{Q}}$. An element of $\rm Pol_{\it\sigma\rm}(\it N_{\mathbb{Q}}\rm)$
is called a polyhedron with \em tailed cone \rm $\sigma$. 
\end{rappel}
Next we introduce the notion of polyhedral divisors
over Dedekind domains.
\begin{definition}
Consider the subset $Z$ of closed points of the 
affine scheme $Y=\rm Spec\,\it A_{\rm 0\it}$.
A $\sigma$-\em polyhedral
divisor $\mathfrak{D}$ over \rm 
$A_{0}$ is a formal sum
\begin{eqnarray*}
\mathfrak{D} = \sum_{z\in Z}\Delta_{z}\cdot z, 
\end{eqnarray*}
where $\Delta_{z}$ belongs to 
$\rm Pol_{\it\sigma\rm}(\it N_{\mathbb{Q}}\rm)$
and $\Delta_{z} = \sigma$ 
for all but finitely many $z$ in $Z$. 
Let
$z_{1},\ldots,z_{r}$ be elements of $Z$ such
that $\{z\in Z\,|\,\Delta_{z} \neq  \sigma\}\subset \{z_{1},\ldots, z_{r}\}$. If the meaning
of $A_{0}$ is clear from the context, then we write
\begin{eqnarray*}
\mathfrak{D} = \sum_{i = 1}^{r}\Delta_{z_{i}}\cdot z_{i}.
\end{eqnarray*}
\end{definition}
In the sequel, we let 
$\omega_{M} = \omega\cap M$ for a polyhedral cone $\omega \subset M_{\mathbb{Q}}$.
Starting with a $\sigma$-polyhedral divisor 
$\mathfrak{D}$ we can build an $M$-graded
algebra over $A_{0}$ with weight cone $\sigma^{\vee}$
in the same way as in [AH06, Section $3$]. 

\begin{rappel}
Let $m\in\sigma^{\vee}$.
Then for any $z\in Z$ the expression
\begin{eqnarray*}
h_{z}(m) = \min_{v\in \Delta_{z}} \,\left\langle m,v \right\rangle
\end{eqnarray*}
is well defined. 
The function $h_{z}$ on the cone $\sigma^{\vee}$ is upper convex and
positively homogeneous. It is identically zero 
if and only if $\Delta_{z} = \sigma$.
The \em evaluation \rm of $\mathfrak{D}$ in 
a vector $m\in\sigma^{\vee}$
is the $\mathbb{Q}$-divisor
\begin{eqnarray*}
\mathfrak{D}(m) = \sum_{z\in Z}h_{z}(m)\cdot z.
\end{eqnarray*}
In analogy with the notation of [FZ03] we
denote by $A_{0}[\mathfrak{D}]$ 
the $M$-graded subring 
\begin{eqnarray*}
\bigoplus_{m\in\sigma^{\vee}_{M}}A_{m}\chi^{m} \subset K_{0}[M],
\,\,\,\rm where\,\,\,\it
A_{m} = H^{\rm 0 \it}\left(Y,\mathcal{O}_{Y}\rm 
\left(\it \lfloor \mathfrak{D}(m)\rfloor \right) \right). 
\end{eqnarray*}
\end{rappel} 
\begin{notation}
Let 
\begin{eqnarray*}
 f = (f_{1}\chi^{m_{1}},\ldots, f_{r}\chi^{m_{r}})
\end{eqnarray*}
be an $r$-tuple of homogeneous elements of $K_{0}[M]$. 
Assume that the vectors $m_{1},\ldots, m_{r}$ generate the 
cone $\sigma^{\vee}$. 
We denote by $\mathfrak{D}[f]$ the $\sigma$-polyhedral divisor
\begin{eqnarray*}
\sum_{z\in Z}\Delta_{z}[f]\cdot z,
\rm \,\,\,\, where\,\,\,\,\it
\rm \Delta_{\it z}\it[f] = \left\{\, v\in N_{\mathbb{Q}}\,|\,
\left\langle m_{i},v \right\rangle\geq 
-\rm ord_{\it z}\it\,f_{i},\,\, i = \rm 1,2,\it\ldots, r\,\right\}.
\end{eqnarray*}
In section $3$, we use a similar notation 
for polyhedral divisors over a regular projective curve.
\end{notation}
The main result of this section is the following theorem.
For a proof of part $\rm (iii)$ we refer the reader to the argument of 
Theorem $2.4$ in [Lan13].
\begin{theorem}
Let $A_{0}$ be a Dedekind domain with field of fractions $K_{0}$ and
let $\sigma\subset N_{\mathbb{Q}}$ be a strongly convex polyhedral 
cone. Then the following hold.
\begin{enumerate}
\item[(i)]
If $\mathfrak{D}$ is a $\sigma$-polyhedral 
divisor over $A_{0}$, then the algebra $A_{0}[\mathfrak{D}]$
is normal, noetherian, and has the same field of fractions
as $K_{0}[M]$.
\item[(ii)] 
Conversely, let 
\begin{eqnarray*}
A = \bigoplus_{m\in\sigma^{\vee}_{M}}A_{m}\chi^{m}
\end{eqnarray*}
be a normal noetherian $M$-graded $A_{0}$-subalgebra of 
$K_{0}[M]$ with weight cone $\sigma^{\vee}$ and $A_{m}\subset K_{0}$,
for all $m\in\sigma^{\vee}_{M}$. 
Assume that the rings $A$
and $K_{0}[M]$ have the same field of fractions\footnote{This condition is equivalent 
to ask that the weight semigroup
of $A$ generates $M$.}.
Then there exists a unique
$\sigma$-polyhedral divisor $\mathfrak{D}$ over $A_{0}$ 
such that $A = A_{0}[\mathfrak{D}]$.
\item[(iii)]
More explicitly, if 
\begin{eqnarray*}
 f = (f_{1}\chi^{m_{1}},\ldots,f_{r}\chi^{m_{r}})
\end{eqnarray*}
is an $r$-tuple of homogeneous elements of $K_{0}[M]$
with nonzero vectors $m_{1},\ldots, m_{r}$ generating 
the lattice $M$, then the normalization of 
the ring
\begin{eqnarray*}
A = A_{0}[f_{1}\chi^{m_{1}},\ldots ,f_{r}\chi^{m_{r}}] 
\end{eqnarray*}
is equal to $A_{0}[\mathfrak{D}[f]]$ \rm (see $2.4$).
\end{enumerate} 
\end{theorem}
Let us give an example related to the ring of integers of a number
field.
\begin{exemple}
For a number field $K$, the group of classes $\rm Cl\,\it K$  
is the quotient of the group of fractional ideals of $K$ by
the subgroup of principal fractional ideals. 
In other words,
$\rm Cl\,\it K \rm = Pic\,\it Y$, where $Y = \rm Spec\,\it \mathbb{Z}_{K}$
 is the affine scheme associated to the ring of integers of $K$. 
It is known that the group $\rm Cl\,\it K$ is finite. 
Furthermore $\mathbb{Z}_{K}$ is
a principal ideal domain if and only if $\rm Cl\,\it K$ is
trivial.

Let $K = \mathbb{Q}(\sqrt{-5})$.
Then $\mathbb{Z}_{K} = \mathbb{Z}[\sqrt{-5}]$ and the group $\rm Cl\,\it K$
is isomorphic to $\mathbb{Z}/2\mathbb{Z}$. A set of
representatives in $\rm Cl\,\it K$ is given by the fractional ideals 
$\mathfrak{a}=(2,\,1+\sqrt{-5})$ and $\mathbb{Z}_{K}$.
Given $x,y$ two independent variables over $K$,
consider the $\mathbb{Z}^{2}$-graded ring
\begin{eqnarray*}
A = \mathbb{Z}_{K}\left[\,3\,x^{2}y,\,2\,y,\,
6\,x\,\right].  
\end{eqnarray*}
Let us describe the normalization $\bar{A}$ of $A$. 
Denoting respectively by $\mathfrak{b}$, 
$\mathfrak{c}$ the prime ideals $(3,\,1+\sqrt{-5})$ and 
$(3,\,1-\sqrt{-5})$, we
have the decompositions 
\begin{eqnarray*}
(2) = \mathfrak{a}^{2},\,\,\,
(3) = \mathfrak{b}\cdot\mathfrak{c}. 
\end{eqnarray*}
Observe that the ideals 
$\mathfrak{a}$, $\mathfrak{b}$, 
$\mathfrak{c}$ are distinct.
Thus we have
\begin{eqnarray*}
\rm div\, 2 = 2\cdot\it[\mathfrak{a}]\,\,\,\,\rm and
\it\,\,\,\,\rm div\,3 = [\it\mathfrak{b}] + [\mathfrak{c}],  
\end{eqnarray*}
where $\mathfrak{a},\mathfrak{b},\mathfrak{c}$
are seen as closed points of 
$Y=\rm Spec\,\it\mathbb{Z}_{K}$.
Let $\mathfrak{D}$ be the polyhedral 
divisor over $\mathbb{Z}_{K}$ 
given by $\Delta_{\mathfrak{a}}\cdot[\mathfrak{a}]+
\Delta_{\mathfrak{b}}\cdot[\mathfrak{b}]+
\Delta_{\mathfrak{c}}\cdot[\mathfrak{c}]$ with the polyhedra
\begin{eqnarray*}
 \Delta_{\mathfrak{a}} = \left\{(v_{1},v_{2})\in\mathbb{Q}^{2}\,|
\,2v_{1}+v_{2}\geq 0,\,v_{2}\geq -2,\,
v_{1}\geq -2\right\}\,\,\,\rm and
\end{eqnarray*}
\begin{eqnarray*}
  \Delta_{\mathfrak{b}} = \Delta_{\mathfrak{c}} = 
\left\{(v_{1},v_{2})\in\mathbb{Q}^{2}\,|
\,2v_{1}+v_{2}\geq -1,\,v_{2}\geq 0,\,
v_{1}\geq -1\right\}.
\end{eqnarray*}
By Theorem $2.5$ we obtain $\bar{A} = A_{0}[\mathfrak{D}]$,
where $A_{0}=\mathbb{Z}_{K}$. The weight cone of $A$
is the first quadrant $\omega = (\mathbb{Q}_{\geq 0})^{2}$. 
An easy computation shows that 
\begin{eqnarray*}
A_{0}[\mathfrak{D}] = 
\mathbb{Z}_{K}\left[\,
2\,y,\,\,6\,xy,\,\, 3(1+\sqrt{-5})\,xy,\,\, 
3\,x^{2}y,\,\,6x\,\right]. 
\end{eqnarray*}
\end{exemple}
The proof of Theorem $2.5$ needs some preparations. 
We start by a well-known result [GY83, Theorem $1.1$] 
yielding an equivalence between noetherian and finitely 
generated properties of multigraded algebras. Note
that this result does not hold for algebras graded by an arbitrary 
abelian group; a counterexample is given in [GY83, $3.1]$.
\begin{lemme}
Let $G$ denote a finitely generated abelian group 
and let $A$ be a $G$-graded ring. 
Then the following statements are equivalent.
\begin{enumerate}
 \item[(i)]
The ring $A$ is noetherian.
 \item[(ii)]
The graded piece $A_{0}$ corresponding to the neutral
element of $G$ is a noetherian ring and the
$A_{0}$-algebra $A$ is finitely generated. 
\end{enumerate}
\end{lemme}
The next lemma will enable us to show that
the ring $A_{0}[\mathfrak{D}]$, coming from a polyhedral divisor $\mathfrak{D}$
over a Dedekind domain $A_{0}$, is noetherian.

\begin{lemme}
For any $\mathbb{Q}$-divisors $D_{1},\ldots, D_{r}$ on 
$Y = \rm Spec\,\it A_{\rm 0\it}$, the $A_{0}$-algebra
\begin{eqnarray*}
B = \bigoplus_{(m_{1},\ldots,m_{r})\in\mathbb{N}^{r}}
H^{0}\left(Y,\mathcal{O}_{Y}
\left(\left\lfloor
\sum_{i = 1}^{r}m_{i}D_{i}\right\rfloor\right)\right) 
\end{eqnarray*}
is finitely generated. 
\end{lemme}
\begin{proof}
Let $d$ be a positive integer 
such that for $i = 1,\ldots,r$,
the divisor
$dD_{i}$ is integral. 
Consider the lattice polytope 
\begin{eqnarray*}
Q= \left\{\,(m_{1},\ldots,m_{r})\in\mathbb{Q}^{r}\,|\,\,
0\leq m_{i}\leq d,\,i = 1,\ldots,r\,\right\}. 
\end{eqnarray*}
The subset $Q\cap\mathbb{N}^{r}$ being finite, 
the $A_{0}$-module 
\begin{eqnarray*}
E: = \bigoplus_{(m_{1},\ldots,m_{r})\in\mathbb{N}^{r}\cap Q}
H^{0}\left(Y,\mathcal{O}_{Y}
\left(\left\lfloor
\sum_{i = 1}^{r}m_{i}D_{i}\right\rfloor\right)\right)
\end{eqnarray*}
is finitely generated (see $1.4$). 
Let $(m_{1},\ldots,m_{r})$ be an element of $\mathbb{N}^{r}$.
Write $m_{i} = dq_{i}+r_{i}$, where $q_{i},r_{i}\in\mathbb{N}$
are such that $0\leq r_{i}<d$. The equality
\begin{eqnarray*}
\left\lfloor \sum_{i = 1}^{r}m_{i}D_{i}\right\rfloor =  
\sum_{i =1}^{r}q_{i}\left\lfloor dD_{i}
\right\rfloor  + \left\lfloor 
\sum_{i = 1}^{r}r_{i}D_{i}
\right\rfloor
\end{eqnarray*}
implies that every homogeneous element of $B$ can 
be expressed as a polynomial
in $E$. If $f_{1},\ldots, f_{s}$ generate
the $A_{0}$-module $E$, then we have $A = A_{0}[f_{1},\ldots,f_{s}]$,
proving our statement.
\end{proof}
Next we give a proof of the first part of Theorem $2.5$.
\begin{proof}
Let $A = A_{0}[\mathfrak{D}]$. Since the cone 
$\sigma^{\vee}$ is full dimensional, by Lemma $1.4$ 
the algebras $A$ and $K_{0}[M]$ have the same field of fractions. 
Let us show that $A$ is a normal ring.
Given a closed point $z\in Z$ and an 
element of $v\in \Delta_{z}$, we define the 
map 
\begin{eqnarray*}
\nu_{z,v}:K_{0}[M]-\{0\}\rightarrow\mathbb{Z} 
\end{eqnarray*}
as follows. Let $\alpha \in K_{0}[M]$ be nonzero. We may decompose
it as a sum of homogeneous elements
\begin{eqnarray*}
\alpha = \sum_{i = 1}^{r}f_{i}\chi^{m_{i}},\,\,\,\rm where\,\,\,\it f_{i}\in K_{\rm 0\it}^{\star}. 
\end{eqnarray*}
We let 
\begin{eqnarray*}
\nu_{z,v}(\alpha) = \min_{1\leq i\leq r}\left\{\rm ord_{\it z}\it\,f_{i} 
\rm + \it\left\langle m_{i},v \right\rangle\right\}. 
\end{eqnarray*}
The map $\nu_{z,v}$ defines a discrete valuation 
on $\rm Frac\,\it A$. 
Denote by $\mathcal{O}_{v,z}$ the associated local ring. 
By the definition of the algebra $A_{0}[\mathfrak{D}]$
 we have
\begin{eqnarray*}
A = K_{0}[M]\cap\bigcap_{z\in Z}\bigcap_{v\in\Delta_{z}}
\mathcal{O}_{v,z}\,.
\end{eqnarray*}
This shows that $A$ is normal as an intersection of normal rings with the same 
field of fractions $\rm Frac\,\it A$. 

It remains to show that $A$ is noetherian.
By Hilbert's Basis Theorem, it suffices to
show that $A$ is finitely 
generated. 
Let $\lambda_{1},\ldots,\lambda_{e}$ be full dimensional
 regular subcones of
$\sigma^{\vee}$, which define a subdivision of $\sigma^{\vee}$. Assume that
for any $i$ the evaluation map
\begin{eqnarray*}
\sigma^{\vee}\rightarrow \rm Div_{\it\mathbb{Q}}(\it Y),\,\,\,
m\mapsto\mathfrak{D}(m)
\end{eqnarray*}
is linear on $\lambda_{i}$. Fix $i\in\mathbb{N}$ such 
that $1\leq i\leq e$. Consider the 
distinct elements $v_{1},\ldots, v_{n}$
of the Hilbert basis of
$\lambda_{i}$. Denote by $A_{\lambda_{i}}$ the algebra
\begin{eqnarray*}
\bigoplus_{m\in\lambda_{i}\cap M}
H^{0}(Y, \mathcal{O}_{Y}(\lfloor \mathfrak{D}(m)\rfloor )\chi^{m}.
\end{eqnarray*}
Then the vectors $v_{1},\ldots,v_{n}$ form a basis
of the lattice $M$ and so 
\begin{eqnarray*}
A_{\lambda_{i}} \simeq 
\bigoplus_{(m_{1},\ldots,m_{n})\in\mathbb{N}^{n}}
H^{0}\left(Y,\mathcal{O}_{Y}\left(\left\lfloor
\sum_{i = 1}^{n}m_{i}\,\mathfrak{D}(v_{i})\right\rfloor\right)\right). 
\end{eqnarray*}
By Lemma $2.8$, the algebra $A_{\lambda_{i}}$ is finitely generated over $A_{0}$.
The surjective map
\begin{eqnarray*}
A_{\lambda_{1}}\otimes\ldots\otimes A_{\lambda_{e}}\rightarrow A 
\end{eqnarray*}
shows that $A$ is also finitely generated.
\end{proof}
For the second part of Theorem $2.5$ we need the following lemma. 
\begin{lemme}
Assume that $A$ verifies the assumptions of $2.5\,\rm (ii)$.
Then the following statements hold.
\begin{enumerate}
\item[(i)] For any $m\in\sigma^{\vee}_{M}$ we have $A_{m}\neq \{0\}$.
In other words, the weight semigroup of the $M$-graded 
algebra $A$ is $\sigma^{\vee}_{M}$.
\item[(ii)] If $L = \mathbb{Q}_{\geq 0}\cdot m'$ is a half-line contained
in $\sigma^{\vee}$, then the ring
\begin{eqnarray*}
A_{L}:= \bigoplus_{m\in L\cap M}A_ {m}\chi^{m} 
\end{eqnarray*}
is normal and noetherian. 
\end{enumerate}
\end{lemme}
\begin{proof}
Let 
\begin{eqnarray*}
 S=\left\{m\in\sigma^{\vee}_{M},\,A_{m}\neq \{0\}\right\}
\end{eqnarray*}
be the weight semigroup of $A$. Assume 
that $S\neq \sigma^{\vee}_{M}$. Then
there exist $e\in\mathbb{Z}_{>0}$ and $m\in M$ such 
that $m\not\in S$ and $e\cdot m\in S$. 
Since $A$ is a noetherian ring, by [GY83, Lemma $2.2$] the 
$A_{0}$-module $A_{em}$
is a fractional ideal of $A_{0}$. By 
Lemma $1.4$ we obtain
\begin{eqnarray*}
A_{em} = H^{0}(Y,\mathcal{O}_{Y}(D_{em})) 
\end{eqnarray*}
for some integral divisor
$D_{em}\in \rm Div_{\it \mathbb{Z}\rm}(\it Y\rm )$.
Let $f$ be a nonzero section of 
\begin{eqnarray*}
H^{0}\left(Y,\mathcal{O}_{Y}\left(
\left\lfloor \frac{D_{em}}{e}\right\rfloor\right)\right).
\end{eqnarray*}
This element exists by virtue of Lemma $1.4$. We have 
the inequalities
\begin{eqnarray*}
\rm div\,\it f^{e}\rm\geq \it 
-e\left\lfloor \frac{D_{em}}{e}\right\rfloor
\rm\geq\it -D_{em}.   
\end{eqnarray*}
The normality
of $A$ implies that $f\in A_{m}$. This contradicts our
assumption and gives $\rm (i)$. 
For the second assertion we notice that $A_{L}$ 
is noetherian by $2.7$ and by the argument of [AH06, Lemma $4.1$]. 

It remains to show that $A_{L}$ is normal.
Let $\alpha\in\rm Frac\,\it A_{L}$ be an integral element
over $A_{L}$. By normality of $A$ and $K_{0}[\chi^{m}]$
we obtain that $\alpha\in A\cap K_{0}[\chi^{m}] = A_{L}$ and so
$A_{L}$ is normal. 
\end{proof}
Let us introduce the following notation.
\begin{notation}
Let 
\begin{eqnarray*}
 (m_{i},e_{i}),\,\,  i = 1,\ldots,r
\end{eqnarray*}
be elements of $M\times \mathbb{Z}$
such that the vectors $m_{1},\ldots,m_{r}$ are nonzero and
generate the lattice
$M$. Then the cone $\omega = \rm Cone(\it m_{\rm 1\it},\ldots, m_{r})$
is full dimensional in $M_{\mathbb{Q}}$.
Consider the $\omega^{\vee}$-polyhedron 
\begin{eqnarray*}
\Delta =  \left\{\, v\in N_{\mathbb{Q}},\,
\left\langle m_{i},v \right\rangle\geq 
-e_{i},\,\, i = 1,2,\ldots, r\,\right\}.
\end{eqnarray*}
 Let 
$L = \mathbb{Q}_{\geq 0}\cdot m$ be a half-line
contained in $\omega$ with a primitive vector $m$.
In other words, the element $m$ generates the semigroup 
$L\cap M$. Denote by $\mathscr{H}_{L}$ the Hilbert basis
in the lattice $\mathbb{Z}^{r}$ 
of the nonzero cone
\begin{eqnarray*}
p^{-1}(L)\cap (\mathbb{Q}_{\geq 0})^{r},\,\,\,\,\rm where 
\,\,\,\,\it p:\mathbb{Q}^{r}\rightarrow M_{\mathbb{Q}}
\end{eqnarray*}
is the $\mathbb{Q}$-linear map sending the canonical
basis onto $(m_{1},\ldots,m_{r})$. We let
\begin{eqnarray*}
 \mathscr{H}_{L}^{\star} =
\left\{\,(s_{1},\ldots,s_{r})\in\mathscr{H}_{L}\,,\,\,\,
\sum_{i = 1}^{r}s_{i}\cdot m_{i}\neq 0\,\right\}. 
\end{eqnarray*}
For any vector $(s_{1},\ldots,s_{r})\in 
\mathscr{H}_{L}^{\star}$ there exists a 
unique $\lambda(s_{1},\ldots,s_{r})\in\mathbb{Z}_{>0}$
such that 
\begin{eqnarray*}
\sum_{i = 1}^{r}s_{i}\cdot m_{i} = \lambda(s_{1},\ldots,s_{r})\cdot m.
\end{eqnarray*}
\end{notation}
The proof of the following lemma uses only some elementary facts 
of commutative algebra and of convex geometry. This lemma is the key point in order
to obtain the Altmann-Hausen presentation of Theorem $2.5\, \rm (ii) $.
\begin{lemme} 
Let $\min\left\langle m,\Delta \right\rangle = 
\min_{v\in\Delta}\langle m,v\rangle$. Under the assumptions of $2.10$
we have
\begin{eqnarray*}
\min\left\langle m,\Delta \right\rangle
 = -\min_{(s_{1},\ldots,s_{r})\,\in\,\mathscr{H}_{L}^{\star}}
\,\frac{\sum_{i = 1}^{r}s_{i}\cdot e_{i}}{
\lambda(s_{1},\ldots, s_{r})}\,.
\end{eqnarray*}
\end{lemme}
\begin{proof} 
Consider the $M$-graded subalgebra 
\begin{eqnarray*}
A = \mathbb{C}[t][t^{e_{1}}\chi^{m_{1}},\ldots,t^{e_{r}}\chi^{m_{r}}]\subset 
\mathbb{C}(t)[M], 
\end{eqnarray*}
where $t$ is a variable.
The field of fractions of $A$ is the same as
that of 
$\mathbb{C}(t)[M]$. By the results of [Hoc72], the normalization of $A$ is 
\begin{eqnarray*}
\bar{A} = \mathbb{C}[\omega_{0}\cap(M\times\mathbb{Z})],\,\,\,\rm  
where\,\,\,\it \omega_{\rm 0\it }\subset M_{\mathbb{Q}}\times\mathbb{Q}
\end{eqnarray*}
is the rational cone generated by $(0,1),(m_{1},e_{1}),\ldots,(m_{r},e_{r})$. 
A routine calculation shows that  
\begin{eqnarray*}
\omega_{0} = \{(w,-\min\left\langle w,\Delta\right\rangle + e)\,|\, w\in\omega, \, e\in\mathbb{Q}_{\geq 0}\},
\end{eqnarray*}
and so 
\begin{eqnarray*}
 \bar{A} = \bigoplus_{m\in\omega\cap M}
H^{0}(\mathbb{A}^{1}_{\mathbb{C}},
\mathcal{O}_{\mathbb{A}^{1}_{\mathbb{C}}}(
\lfloor
\min\left\langle m,\Delta\right\rangle
\rfloor \cdot(0)))\chi^{m},
\end{eqnarray*}
where $\mathbb{A}^{1}_{\mathbb{C}}=\rm Spec\,\it 
\mathbb{C}[t]$.

The sublattice $G\subset M$ generated by 
$p(\mathscr{H}_{L}^{\star})$ is a subgroup
of $\mathbb{Z}\cdot m$. Therefore there exists 
a unique integer 
$d\in\mathbb{Z}_{>0}$ such
that $G = d\,\mathbb{Z}\cdot m$.
For an element $m'\in\omega\cap M$, we denote by $A_{m'}$ 
(resp. $\bar{A}_{m'}$) the graded
piece of $A$ (resp. $\bar{A}$) corresponding to $m'$. 
Then the 
normalization $\bar{A}^{(d)}_{L}$ of the algebra 
\begin{eqnarray*}
A^{(d)}_{L} : = 
\bigoplus_{s\geq 0}A_{sdm}\chi^{sdm}
\rm \,\,\,is\,\,\, \it 
B_{L} := \bigoplus_{s\geq \rm 0\it}
\bar{A}_{sdm}\chi^{sdm}.
\end{eqnarray*} 
Furthermore  
\begin{eqnarray*}
A_{L} =\bigoplus_{s\geq 0}A_{sm}\chi^{sm} 
\end{eqnarray*}
is generated over $\mathbb{C}[t]$
by the elements 
\begin{eqnarray*}
f_{(s_{1},\ldots,s_{r})} :=
\prod_{i = 1}^{r}(t^{e_{i}}\chi^{m_{i}})^{s_{i}} = 
t^{\sum_{i = 1}^{r}s_{i}e_{i}}\chi^{\lambda(s_{1},\ldots, s_{r})m},
\end{eqnarray*}
where $(s_{1},\ldots,s_{r})$ runs over $\mathscr{H}_{L}^{\star}$.
By the choice of the integer $d$ we have 
$A^{(d)}_{L} = A_{L}$. Considering the
$G$-graduation of $A^{(d)}_{L}$,
for any $(s_{1},\ldots,s_{r})\in
\mathscr{H}_{L}^{\star}$ 
the element
$f_{(s_{1},\ldots,s_{r})}$ of the graded ring $A^{(d)}_{L}$
has degree
\begin{eqnarray*}
\rm deg\,\it f_{
(s_{\rm 1\it},\ldots, s_{r})} := 
\frac{\lambda(
s_{\rm 1\it},\dots, s_{r})}{d}\,\,.
\end{eqnarray*}
Letting 
\begin{eqnarray*}
 D = -\min_{(s_{1},\ldots,s_{r})
\in\mathscr{H}_{L}^{\star}}
\frac{\rm div\it\,
f_{(s_{\rm 1\it},\ldots,s_{r})}}
{\rm deg\,\it f_{
(s_{\rm 1\it},\ldots, s_{r})}}
=
-\min_{(s_{1},\ldots,s_{r})\in\mathscr{H}_{L}^{\star}}
d\cdot \frac{\sum_{i = 1}^{r}s_{i}e_{i}}{
\lambda(s_{1},\ldots, s_{r})}\cdot (0),
\end{eqnarray*} 
by Lemma $1.7$ we obtain 
\begin{eqnarray*}
\bar{A}^{(d)}_{L} = 
\bigoplus_{s\geq 0}
H^{0}
(\mathbb{A}^{1}_{\mathbb{C}},
\mathcal{O}_{\mathbb{A}^{1}_{\mathbb{C}}}
(\lfloor sD\rfloor))\chi^{sdm}.
\end{eqnarray*}
The equality $\bar{A}^{(d)}_{L} = B_{L}$
implies that for any integer $s\geq 0$
\begin{eqnarray*}
 H^{0}(\mathbb{A}^{1}_{\mathbb{C}},
\mathcal{O}_{\mathbb{A}^{1}_{\mathbb{C}}}(
\lfloor
\min\left\langle sd\cdot m,\Delta\right\rangle
\rfloor \cdot(0))) =  H^{0}
(\mathbb{A}^{1}_{\mathbb{C}},
\mathcal{O}_{\mathbb{A}^{1}_{\mathbb{C}}}(\lfloor sD\rfloor)).
\end{eqnarray*}
Hence by Lemma $1.6$ we have
\begin{eqnarray*}
D = 
\min\,\langle d\cdot m,\Delta\rangle
\cdot (0). 
\end{eqnarray*}
Dividing by $d$, we obtain the desired formula.
\end{proof}
Let $A$ be an $M$-graded algebra satisfying the assumptions
of $2.5\,\rm (ii)$. Using the D.P.D. presentation
on each half line of the weight cone $\sigma^{\vee}$ (see Lemma $1.6$), 
we can define a map
\begin{eqnarray*}
\sigma^{\vee}\rightarrow \rm Div_{\mathbb{Q}}(\it Y\rm),\,\,\,
\it m\mapsto D_{m}. 
\end{eqnarray*}
It is
upper convex, positively homogeneous, and verifies,
for any $m\in\sigma^{\vee}_{M}$, the equality
\begin{eqnarray*}
A_{m} = H^{0}(C,\mathcal{O}_{C}(\lfloor D_{m}\rfloor)). 
\end{eqnarray*}
By Lemma $2.11$, this map is piecewise 
linear (see [AH06, $2.11$]). Equivalently, 
$m\mapsto D_{m}$ is the evaluation map of a polyhedral divisor.
The following proof will specify this idea.
\begin{proof}[Proof of $2.5\,\rm (ii)$]
By $2.7$ we may consider a system of homogeneous generators 
\begin{eqnarray*}
f = (f_{1}\chi^{m_{1}},\ldots,f_{r}\chi^{m_{r}}) 
\end{eqnarray*}
of $A$, with 
nonzero vectors $m_{1},\ldots,m_{r}\in M$. We use
the same notation as in $2.4$. Denote by
$\mathfrak{D}$ the $\sigma$-polyhedral 
divisor $\mathfrak{D}[f]$.
Let us show that $A = A_{0}[\mathfrak{D}]$.
Let $L = \mathbb{Q}_{\geq 0}\cdot m$ be a half-line 
contained in $\omega = \sigma^{\vee}$, with $m$
being the primitive vector of $L$. 
By Lemma $2.9$, the 
graded subalgebra
\begin{eqnarray*}
A_{L}:=\bigoplus_{m'\in L\cap M}A_{m'}\chi^{m'}\subset K_{0}[\chi^{m}]  
\end{eqnarray*}
is normal, noetherian, and has the same field of fractions as 
that of $K_{0}[\chi^{m}]$. 
Furthermore, with the same notation as in 
$2.10$, the algebra $A_{L}$ is generated by the set
\begin{eqnarray*}
\left\{\,
\prod_{i = 1}^{r}(f_{i}\chi^{m_{i}})^{s_{i}},\,\,
(s_{1},\ldots,s_{r})\in\mathscr{H}_{L}^{\star}\,
\right\}. 
\end{eqnarray*}
By Lemma $1.7$, if 
\begin{eqnarray*}
D_{m} : = -\min_{(s_{1},\ldots,s_{r})
\,\in\,\mathscr{H}_{L}^{\star}}
\,\frac{\sum_{i=1}^{r}s_{i}\,\rm div\,\it f_{i}}
{\lambda(s_{\rm 1\it},\ldots, s_{r})}, 
\end{eqnarray*}
then $A_{L} = A_{0}[D_{m}]$ with respect to 
the variable $\chi^{m}$. 
By Lemma $2.11$, for any closed point $z\in Z$ we have
\begin{eqnarray*}
h_{z}[f](m) = \min\langle m, \Delta_{z}[f]\rangle =  
-\min_{(s_{1},\ldots,s_{r})
\,\in\,\mathscr{H}_{L}^{\star}}
\,\frac{\sum_{i=1}^{r}s_{i}\,\rm ord_{\it z}\it\,f_{i}}
{\it \lambda(s_{\rm 1}\it,\ldots, s_{r})}\,.
\end{eqnarray*}
Hence $\mathfrak{D}(m) = D_{m}$. 
Since this equality holds for all primitive 
vectors belonging to $\sigma^{\vee}$, 
we may conclude that $A = A_{0}[\mathfrak{D}]$.
The uniqueness of $\mathfrak{D}$ is straightforward (see Lemma $1.4$
and [Lan13, $2.2$]).
\end{proof}
Using well-known facts about Dedekind domains we obtain the following result.
\begin{proposition}
Let $A_{0}$ be a Dedekind domain and let $B_{0}$ be the integral closure of
$A_{0}$ in a finite separable extension $L_{0}/K_{0}$, where $K_{0} = \rm Frac\,\it A_{\rm 0}$.
Let $\mathfrak{D} = \sum_{z\in Z}\Delta_{z}\cdot z$ be a polyhedral divisor over $A_{0}$,
where $Z\subset Y = \rm Spec\,\it A_{\rm 0}$ is the subset of closed points. Let
$Y' = \rm Spec\,\it B_{\rm 0}$ and consider the natural projection $p:Y'\rightarrow Y$. Then
$B_{0}$ is a Dedekind domain and we have the formula
\begin{eqnarray*}
A_{0}[\mathfrak{D}]\otimes_{A_{0}}B_{0} = B_{0}[p^{\star}\mathfrak{D}]\rm\,\,\, with
\,\,\,\it p^{\star}\mathfrak{D} = \sum_{z\in Z}\rm \Delta_{\it z}\it\cdot p^{\star}(z). 
\end{eqnarray*}
\end{proposition}
\begin{exemple}
Consider the polyhedral divisor 
\begin{eqnarray*}
 \mathfrak{D} = \Delta_{(t)}\cdot (t) + \Delta_{(t^{2}+1)}\cdot (t^{2}+1)
\end{eqnarray*}
over the Dedekind ring $A_{0} = \mathbb{R}[t]$, where the coefficients are     
\begin{eqnarray*}
\Delta_{(t)} = (-1,0) + \sigma,\,\,\, \Delta_{(t^{2}+1)} = [(0,0), (0,1)] + \sigma,
\end{eqnarray*}
and where $\sigma\subset \mathbb{Q}^{2}$ is the rational cone generated by $(1,0)$ and $(1,1)$.
An easy computation shows that 
\begin{eqnarray*}
A_{0}[\mathfrak{D}] = \mathbb{R}\left[t, -t\chi^{(1,0)}, \chi^{(0,1)}, t(t^{2}+1)\chi^{(1,-1)}\right]\simeq 
\frac{\mathbb{R}[x_{1},x_{2},x_{3},x_{4}]}{((1 + x^{2}_{1})x_{2} + x_{3}x_{4})},
\end{eqnarray*}
where $x_{1},x_{2},x_{3},x_{4}$ are independent variables over $\mathbb{R}$. 
Let $\zeta = \sqrt{-1}$. Considering the natural projection 
$p:\mathbb{A}^{1}_{\mathbb{C}}\rightarrow \mathbb{A}^{1}_{\mathbb{R}}$
we obtain
\begin{eqnarray*}
p^{\star}\mathfrak{D} = \Delta_{0}\cdot 0 + \Delta_{(t^{2}+1)}\cdot \zeta + \Delta_{(t^{2}+1)}\cdot (-\zeta).
\end{eqnarray*}
Letting $B_{0} = \mathbb{C}[t]$ one concludes that 
$A_{0}[\mathfrak{D}]\otimes_{\mathbb{R}}\mathbb{C} = B_{0}[p^{\star}\mathfrak{D}]$.
\end{exemple}

\section{Multigraded algebras and algebraic function fields}
In this section, we study another type of multigraded algebras.
They are described by a proper polyhedral divisor over an
algebraic function field in one variable. Fix an arbitrary field $k$.
Recall that an \em algebraic function field \rm (in one variable)
over $k$ is a finitely generated field extension $K_{0}/k$
of transcendence degree one with the property that $k$ is algebraically closed in $K_{0}$.
\begin{rappel}
By virtue of our convention, a regular projective curve $C$ over 
$k$ yields an algebraic function field $K_{0}/k$, where $K_{0} = k(C)$. 
As an application of the valuative criterion
of properness (see [EGA II, Section $7.4$]), every algebraic function field $K_{0}/k$
is the field of rational functions of a unique (up to isomorphism)
regular projective curve $C$  over $k$. 
\end{rappel}
In the next paragraph, we recall the 
construction of the curve $C$ starting from an algebraic function field $K_{0}$.
\begin{rappel}
A \em valuation ring \rm of $K_{0}$ is 
a proper subring $\mathcal{O} \subset K_{0}$
strictly containing $k$ and such that
for any nonzero element $f\in K_{0}$, either 
$f\in \mathcal{O}$ or $\frac{1}{f}\in\mathcal{O}$. 
By [Sti93, $1.1.6$]
every valuation ring of $K_{0}$ is the
ring associated to a discrete valuation of 
$K_{0}/k$. A subset $P\subset K_{0}$
is called a \em place \rm of $K_{0}$
if there is some valuation ring $\mathcal{O}$
of $K_{0}$ such that $P$ is the
maximal ideal of $\mathcal{O}$. We denote
by $\mathscr{R}_{k}\, K_{0}$ the 
set of places of $K_{0}$. The latter is called 
the \em Riemann surface \rm of $K_{0}$.
By [EGA II, $7.4.18$] the set 
$\mathscr{R}_{k}\, K_{0}$ can be 
identified with the (closed) points of a regular projective curve $C$
over the field $k$
such that $K_{0} = k(C)$.
\end{rappel}
In the sequel we consider 
$C =\mathscr{R}_{k}\, K_{0}$ 
as a geometrical object with its
structure of scheme. By convention an element $z$
belonging to $C$ is a closed point.
We write $P_{z}$ the 
place associated to a point $z\in C$. 
Note that we keep the notation of convex geometry introduced 
in $2.1$.
\begin{rappel}
Let $M,N$ be a pair of dual lattices, and
let $\sigma\subset N_{\mathbb{Q}}$ be a strongly 
convex polyhedral cone. 
A \em $\sigma$-polyhedral divisor \rm over $K_{0}$ (or over $C$) is a formal sum
$\mathfrak{D} = \sum_{z\in C}\Delta_{z}\cdot z$
with $\Delta_{z} \in\rm Pol_{\it \sigma}(\it N_{\mathbb{Q}})$ 
and $\Delta_{z} = \sigma$ for all but finitely many $z\in C$.
Again we let
\begin{eqnarray*}
\mathfrak{D}(m) = \sum_{z\in C}\min_{v\in\Delta_{z}}\langle m, v\rangle\cdot z 
\end{eqnarray*}
be the \em evaluation \rm in $m\in\sigma^{\vee}$; that is, a $\mathbb{Q}$-divisor over the 
curve $C$. 
We let $\kappa(P) = \mathcal{O}/P$, where 
$\mathcal{O}$ is the valuation ring of a place $P$. 
The field $\kappa(P)$
is a finite extension of $k$ 
[Sti93, $1.1.15$]; we call it the \em residue field \rm of $P$.
We denote by $[\kappa(P):k]$ the dimension of the 
$k$-vector space $\kappa(P)$; the latter is also called the degree of the place $P$.
Recall that we define the \emph{degree} of a $\mathbb{Q}$-divisor $D = \sum_{z\in C}a_{z}\cdot z$
to be the rational number $${\rm deg}\,D = \sum_{z\in C}[\kappa(P_{z}):k]\cdot a_{z}.$$
Similarly, the \em degree \rm of the polyhedral divisor $\mathfrak{D}$ is the 
Minkowski sum
\begin{eqnarray*}
\rm deg\it\, \mathfrak{D} = 
\sum_{z\in C}[\kappa(P_{z}):k]\cdot
\rm\Delta_{\it z}. 
\end{eqnarray*}
Given $m\in\sigma^{\vee}$ we have naturally the
relation 
$\it (\rm deg\it\, \mathfrak{D} )(m) = 
\rm deg\it\, \mathfrak{D}(m)$. 
\end{rappel}
We can now introduce the notion of properness for polyhedral divisors (see [AH06, $2.7$, $2.12$]).
\begin{definition}
A $\sigma$-polyhedral divisor $\mathfrak{D} = 
\sum_{z\in C}\Delta_{z}\cdot z$ is called 
\em proper \rm if it satisfies the following
conditions.
\begin{enumerate}
 \item[(i)] The polyhedron $\rm deg\it\, \mathfrak{D}$
is strictly contained in the cone $\sigma$.
\item[(ii)] If $\rm deg\it\, \mathfrak{D}(m) \rm = 0$,
then $m$ belongs to the boundary of $\sigma^{\vee}$ and
a multiple of $\mathfrak{D}(m)$ is principal. 
\end{enumerate}
\end{definition}
Our next main result gives a description similar
to that in $2.5$ for algebraic function fields.
For a proof of $3.5\,\rm (iii)$ we refer to the argument of
[Lan13, $2.4$].
\begin{theorem}
Let $k$ be a field, and let 
$C = \mathscr{R}_{k}\, K_{0}$ be the Riemann
surface of an algebraic function field $K_{0}/k$. 
Then the following statements hold.
\begin{enumerate}
 \item[(i)]Let 
\begin{eqnarray*}
A = \bigoplus_{m\in\sigma^{\vee}_{M}}
A_{m}\chi^{m} 
\end{eqnarray*}
be an $M$-graded normal noetherian $k$-subalgebra of 
$K_{0}[M]$ with weight cone $\sigma^{\vee}$ and $A_{0} = k$.
We assume that for all $m\in\sigma^{\vee}_{M}$, $A_{m}\subset K_{0}$.
If $A$ and $K_{0}[M]$ have the same field of fractions,
then there exists a unique proper
$\sigma$-polyhedral divisor $\mathfrak{D}$ over $C$ 
such that $A = A[C,\mathfrak{D}]$, where
\begin{eqnarray*}
A[C,\mathfrak{D}] = \bigoplus_{m\in\sigma^{\vee}_{M}}
H^{0}(C,\mathcal{O}_{C}(\lfloor \mathfrak{D}(m)\rfloor))\chi^{m}. 
\end{eqnarray*}
\item[(ii)]
Let $\mathfrak{D}$ be a proper $\sigma$-polyhedral divisor over $C$.
Then the algebra $A[C,\mathfrak{D}]$ is $M$-graded, normal, and
finitely generated
with weight cone $\sigma^{\vee}$. Furthermore it
 has the same field of fractions as $K_{0}[M]$.
\item[(iii)]  
Let
\begin{eqnarray*}
A = k[f_{1}\chi^{m_{1}},\ldots, f_{r}\chi^{m_{r}}] 
\end{eqnarray*}
be an $M$-graded subalgebra of $K_{0}[M]$, where the
 $f_{i}\chi^{m_{i}}$ is a homogeneous element of nonzero degree $m_{i}$. 
Let $f = (f_{1}\chi^{m_{1}},\ldots, f_{r}\chi^{m_{r}} )$.
Assume that $A$ and $K_{0}[M]$ have the same field of fractions.
Then $\mathfrak{D}[f]$ is the proper $\sigma$-polyhedral
divisor such that the normalization of $A$ is $A[C,\mathfrak{D}[f]]$ \rm (see $2.4$).
\end{enumerate}
\end{theorem}
For the proof of $3.5$ we need some preliminary results. We begin by collecting some properties of
an $M$-graded algebra $A$ as in $3.5\,\rm (i)$ to some graded subring $A_{L}$.

\begin{lemme}
Let $A$ be an $M$-graded algebra satisfying 
the assumptions of $3.5\,\rm (i)$. 
Given a half-line $L = \mathbb{Q}_{\geq 0}\cdot m\subset \sigma^{\vee}$ with
a primitive vector $m$, consider the subalgebra
\begin{eqnarray*}
A_{L}  = \bigoplus_{m'\in L\cap M}A_{m'}\chi^{m'}. 
\end{eqnarray*}
Let
\begin{eqnarray*}
Q(A_{L})_{0}  = \left\{\,\frac{a}{b}\,|\,\,\, 
a\in A_{sm},\, b\in A_{sm},\, b\neq 0,
\,s\geq 0 \,\right\}. 
\end{eqnarray*}
Then the following assertions hold.
\begin{enumerate}
\item[(i)] The algebra $A_{L}$ is finitely generated and normal.
\item[(ii)] Either $Q(A_{L})_{0} = k$ or $Q(A_{L})_{0}  = K_{0}$.
\item[(iii)] If $Q(A_{L})_{0} = k$, then 
$A_{L} = k[\beta \chi^{dm}]$ for some $\beta\in K_{0}^{\star}$
and some $d\in\mathbb{Z}_{>0}$.   
\end{enumerate}
\end{lemme}
\begin{proof}
The proof of $(i)$ is similar to that of $2.9\,\rm (ii)$ and so we
omitted it.

The field $Q(A_{L})_{0}$ is an extension of
$k$ contained in $K_{0}$. If the transcendence
degree of $Q(A_{L})_{0}$ over $k$ is zero,
then by normality of $A_{L}$
 we have $Q(A_{L})_{0} = k$.
Otherwise the extension $K_{0}/Q(A_{L})_{0}$ is algebraic. Let $\alpha$
be an element of $K_{0}$.
Then there exist $a_{1},\ldots, a_{d}\in Q(A_{L})_{0}$
with $a_{d}\neq 0$ such that 
\begin{eqnarray*}
\alpha^{d} = \sum_{j = 1}^{d}a_{j}\alpha^{d-j}. 
\end{eqnarray*}
Let
\begin{eqnarray*}
I = \{i\in\{1,\ldots,d\},\,\,a_{i}\neq 0\}. 
\end{eqnarray*}
For any $i\in I$ we write $a_{i} = \frac{p_{i}}{q_{i}}$
with $p_{i}, q_{i}\in A_{L}$ being homogeneous of the same
degree. Considering $q =\prod_{i\in I}q_{i}$ we obtain
the equality
\begin{eqnarray*}
(\alpha q)^{d} = \sum_{j = 1}^{d}a_{j}q^{j}(\alpha q)^{d-j}. 
\end{eqnarray*}
The normality of $A_{L}$ gives $\alpha q\in A_{L}$. Thus $\alpha = \alpha q/q\in Q(A_{L})_{0}$. 

To show $\rm (iii)$ we let $S\subset\mathbb{Z}\cdot m$ be the
weight semigroup of the graded algebra $A_{L}$. Since $L$ is contained
in the weight cone $\sigma^{\vee}$, $S$ is nonzero. 
Therefore if $G$ is the subgroup
generated by $S$, then there exists $d\in\mathbb{Z}_{>0}$ such
that $G = \mathbb{Z}\,d\cdot m$. 
Letting $u = \chi^{dm}$ we can
write
\begin{eqnarray*}
A_{L} = \bigoplus_{s\geq 0} A_{sdm}u^{s}. 
\end{eqnarray*}
Thus for any homogeneous elements $a_{1}u^{l}, a_{2}u^{l}\in A_{L}$
of the same degree 
we have $\frac{a_{1}}{a_{2}}\in Q(A_{L})^{\star}_{0} = k^{\star}$
so that
\begin{eqnarray*}
A_{L} = \bigoplus_{s\in S'}kf_{s}\, u^{s}, 
\end{eqnarray*}
where $S':=\frac{1}{d}\,S$ and $f_{s}\in k(C)^{\star}$.
Let us fix homogeneous
generators $f_{s_{1}}u^{s_{1}},\ldots ,f_{s_{r}}u^{s_{r}}$ of the 
$G$-graded algebra $A_{L}$. Consider $d' := \rm g.c.d\it (s_{\rm 1\it},\ldots, s_{r})$. 
If $d'>1$, then the inclusion $S\subset dd'\,\mathbb{Z}\cdot m$ yields a 
contradiction. So $d' = 1$ and there are
integers $l_{1},\ldots,l_{r}$ such that $1 = \sum_{i = 1}^{r}l_{i}s_{i}$. The element
\begin{eqnarray*}
 \beta u =\prod_{i = 1}^{r}(f_{s_{i}}u^{s_{i}})^{l_{i}}
\end{eqnarray*}
verifies 
\begin{eqnarray*}
 \frac{(\beta u)^{s_{1}}}{f_{s_{1}}u^{s_{1}}}\in Q(A_{L})^{\star}_{0} 
= k^{\star}.
\end{eqnarray*}
By normality of $A_{L}$, $\beta u\in A_{L}$ and so $A_{L} = 
k[\beta u] = k[\beta \chi^{dm}]$. Now 
$\rm (iii)$ follows.  
\end{proof}
The following lemma is well known. For the main argument we refer 
the reader to [Dem88, Section $3$], [Liu02, \S 7.4.1, Proposition 4.4], or [AH06, $9.1$].
\begin{lemme}
Let $D_{1},D_{2}, D$ be $\mathbb{Q}$-divisors on $C$. Then the following hold.
\begin{enumerate}
 \item[\rm (i)] If $D$ has positive degree, then there
exists $d\in\mathbb{Z}_{>0}$ 
such that the invertible sheaf $\mathcal{O}_{C}(\lfloor dD\rfloor)$ of $\mathcal{O}_{C}$-modules 
is very ample. Furthermore, the graded algebra
\begin{eqnarray*}
B = \bigoplus_{l\geq 0}H^{0}(C,\mathcal{O}_{C}(\lfloor lD\rfloor))t^{l}, 
\end{eqnarray*}
where $t$ is a variable over $k(C)$, is finitely generated. The field
of fractions of $B$ is $k(C)(t)$.
\item[\rm (ii)] Assume that
for $i = 1,2$ we have either $\rm deg\,\it D_{i}
\rm >0$ or $rD_{i}$ is principal for some $r\in\mathbb{Z}_{>0}$.
If for any $s\in\mathbb{N}$ the inclusion
\begin{eqnarray*}
H^{0}(C,
\mathcal{O}_{C}(\lfloor sD_{1}\rfloor))\subset  H^{0}(C,
\mathcal{O}_{C}(\lfloor sD_{2}\rfloor))
\end{eqnarray*}
holds, then we have $D_{1}\leq D_{2}$.
\end{enumerate}
\end{lemme}

In the next corollary, we keep the notation of Lemma $3.6$.
Using Demazure's Theorem for normal graded algebras, we show
that each $A_{L}$ admits a D.P.D. presentation with
the same regular projective curve $C$. 
\begin{corollaire}
There exists a unique $\mathbb{Q}$-divisor $D$ on $C$ such that
\begin{eqnarray*}
A_{L} = \bigoplus_{s\geq 0}H^{0}(C,
\mathcal{O}_{C}(\lfloor sD\rfloor))\chi^{sm} 
\end{eqnarray*}
and the following hold.
\begin{enumerate}
 \item[(i)] If $Q(A_{L})_{0} = k$, then
$D = \frac{\rm div\,\it f}{\it d}$ for some $f\in K_{0}^{\star}$
and some $d\in\mathbb{Z}_{>0}$. 
\item[(ii)] If $Q(A_{L})_{0} = K_{0}$, then $\rm deg\,\it D\rm >0$.
\item[(iii)] If $f_{1}\chi^{s_{1}m},\ldots, f_{r}\chi^{s_{r}m}$
are homogeneous generators of the algebra $A_{L}$, then 
\begin{eqnarray*}
D = -\min_{1\leq i\leq r}\frac{\rm div\,\it f_{i}}{s_{i}}\,. 
\end{eqnarray*}
\end{enumerate}
\end{corollaire}
\begin{proof}
$\rm (i)$ Assume that $Q(A_{L})_{0} = k$. By Lemma $3.6$,
$A_{L} = k[\beta\,\chi^{dm}]$
for some $\beta\in K_{0}^{\star}$ and some $d\in\mathbb{Z}_{>0}$.
Thus, we can take $D = \frac{\rm div\,\beta^{-1}}{\it d}$. 
The uniqueness in this case is easy. This gives assertion $\rm (i)$.

$\rm (ii)$ The field of rational functions of 
the normal variety $\rm Proj\,\it A_{L}$ is
$K_{0} = Q(A_{L})_{0}$. Since $\rm Proj\,\it A_{L}$ is a regular
projective curve over $A_{0} = k$, we may identify its points
with the places of $K_{0}$. Therefore the existence and 
the uniqueness of $D$ follow from Demazure's 
Theorem (see [Dem88, Theorem 3.5]).
Furthermore $Q(A_{L})_{0}\neq k$ implies that
$\rm dim_{\it k}\it A_{\it sm}\rm\geq 2,$ for
some $s\in\mathbb{Z}_{>0}$. Hence by [Sti93, 1.4.12]
we obtain $\rm deg\, \it D\rm >0$.
  
The proof of $\rm (iii)$ follows from $3.7$ and from the argument in [FZ03, $3.9$]. 
\end{proof}
As a consequence of Corollary $3.8$, again we can apply the formula of $2.11$
to obtain the existence of the polyhedral divisor $\mathfrak{D}$ as in the statement of $3.5\,\rm (i)$.
 
\begin{proof}[Proof of $3.5\,\rm (i)$]
Let us adopt the notation introduced in $2.4$ and
$\,2.10$.
Let 
\begin{eqnarray*}
f = (f_{1}\chi^{m_{1}},\ldots, f_{r}\chi^{m_{r}})
\end{eqnarray*}
be a system of homogeneous generators of $A$. 
Consider a half-line 
\begin{eqnarray*}
L = \mathbb{Q}_{\geq 0}\cdot m\subset \sigma^{\vee}
\end{eqnarray*}
with primitive vector $m\in M$. By Corollary $3.8$ 
\begin{eqnarray*}
A_{L} = \bigoplus_{s\geq 0}H^{0}(C,
\mathcal{O}_{C}(\lfloor sD_{m}\rfloor))\chi^{sm} 
\end{eqnarray*}
for a unique $\mathbb{Q}$-divisor $D_{m}$ on $C$. By the proof of
[AH06, Lemma $4.1$] the algebra $A_{L}$
is generated by 
\begin{eqnarray*}
\prod_{i = 1}^{r}(f_{i}\chi^{m_{i}})^{s_{i}},\,\,\,\rm where\,\,\,\it  (s_{\rm 1\it},\ldots s_{r})
\in\mathscr{H}_{L}^{\star}.
\end{eqnarray*}
By Corollary $3.8\,\rm (iii)$ and Lemma $2.11$
we have $\mathfrak{D}[f](m) = D_{m}$ and so 
$A = A[C,\mathfrak{D}[f]]$. 

It remains to show that $\mathfrak{D} = \mathfrak{D}[f]$
is proper; the uniqueness of $\mathfrak{D}$ will then follow by 
Lemma $3.7\,\rm (ii)$. Denote by $S\subset C$ the union of the supports
of the divisors $\rm div\,\it f_{i}$, for $i = 1,\ldots, r$.
Let $v\in \rm deg\,\it \mathfrak{D}$. We can write
\begin{eqnarray*}
v = \sum_{z\in S}[\kappa(P_{z}):k]\cdot v_{z} 
\end{eqnarray*}
for some $v_{z}\in\Delta_{z}[f]$.
Therefore for any $i$ we have 
\begin{eqnarray*}
\langle m_{i}, \sum_{z\in S}[\kappa(P_{z}):
k]\cdot v_{z}\rangle \geq -\sum_{s\in S}[\kappa(P_{z}):k]\cdot
\rm ord_{\it z}\,\it f_{i} =-\rm deg\, div\,\it f_{i}\rm  = 0 
\end{eqnarray*}
and so $\rm deg\,\it \mathfrak{D}\subset \sigma.$
If $\rm deg\,\it \mathfrak{D} = \sigma$, then one concludes that $\rm Frac\,\it A$
is different from $\rm Frac\,\it K_{\rm 0\it}[M]$,
contradicting our assumption.
Hence $\rm deg\,\it \mathfrak{D}\neq \sigma$.
Let $m\in\sigma^{\vee}_{M}$ be such that
$\rm deg\,\it\mathfrak{D}(m)\rm = 0$.
Then $m$ belongs to the boundary of 
$\sigma^{\vee}$.
Consider the half-line $L$ generated by $m$.
Applying Corollary $3.8\,\rm (i)$ to the algebra $A_{L}$,
we deduce that a multiple of
$\mathfrak{D}(m)$ is principal, proving that $\mathfrak{D}$ is proper. 
\end{proof}
\begin{proof}[Proof of $3.5 \,\rm (ii)$]
Let us show that the algebras
$A = A[C,\mathfrak{D}]$ and $K_{0}[M]$ have the same field of fractions. Let 
$L = \mathbb{Q}_{\geq 0}\cdot m$ be a half-line intersecting $\sigma^{\vee}$
with its relative interior and having $m$ for primitive vector. Since
$\rm deg\,\it \mathfrak{D}(m)\,\rm >0$ by Lemma $3.7\,\rm (i)$ 
we have $\rm Frac\,\it A_{L} = K_{\rm 0}(\it\chi^{m})$, yielding our first claim. 
As a consequence, $\sigma^{\vee}$ is the weight cone of the $M$-graded algebra $A$. 
The proof of the normality is similar to that of $2.5\,\rm (i)$.

Let us show further that $A$ is finitely generated. First we may consider a subdivision 
of $\sigma^{\vee}$ by regular strongly convex polyhedral
cones $\omega_{1},\ldots, \omega_{s}$ such that for any $i$ we have 
$\omega_{i}\cap \rm relint\,\it\sigma^{\vee} \neq \emptyset$, $\omega_{i}$ is full dimensional,
and $\mathfrak{D}$ is linear on $\omega_{i}$. Fix $1\leq i\leq s$ and a positive integer $\mu$. 
Let $(e_{1},\ldots, e_{n})$
be a basis of $M$ generating the cone $\omega_{i}$ and such that $e_{1}\in\rm relint\,\it\sigma^{\vee}$.
By properness of $\mathfrak{D}$, there exists $d\in\mathbb{Z}_{>0}$ such that every $\mathfrak{D}(de_{j})$ is a 
globally generated integral divisor. Letting
 \begin{eqnarray*}
A_{\omega_{i},\mu} = \bigoplus_{(a_{1},\ldots, a_{n})\in\mathbb{Z}^{n}}H^{0}\left(C,
\mathcal{O}\left(\left\lfloor \sum_{i = 1}^{n}a_{i}\mu e_{i}\right\rfloor\right)\right)\chi^{\sum_{i}a_{i}\mu e_{i}} 
\end{eqnarray*}
we consider homogeneous elements $f_{1}\chi^{m_{1}},\ldots, f_{r}\chi^{m_{r}}\in A_{\omega_{i},d}$
obtained by taking generators of the space of global sections of every $\mathcal{O}(\mathfrak{D}(de_{j}))$
and homogeneous generators of the graded algebra
\begin{eqnarray*}
B = \bigoplus_{l\geq 0}H^{0}(C,\mathcal{O}_{C}(\mathfrak{D}(dle_{1})))\chi^{lde_{1}}, 
\end{eqnarray*}
see Lemma $3.7\,\rm (i)$. Using Theorem $3.5\,\rm (iii)$ the normalization of
$k[f_{1}\chi^{m_{1}},\ldots, f_{r}\chi^{m_{r}}]$ is $A_{\omega_{i},d}$.
So by Theorem $2$ in [Bou72, V$3.2$]
the algebra $A_{\omega_{i}} = A_{\omega_{i},1}$ is finitely generated. One concludes
by considering the surjection 
$A_{\omega_{1}}\otimes\ldots\otimes A_{\omega_{s}}\rightarrow A$.
\end{proof}
In the next assertion, we study how the algebra associated to a polyhedral
divisor over a regular projective curve changes, when we extend the scalars passing 
to the algebraic closure of the ground field $k$.
The first claims are classical for the theory of
algebraic function fields, and so the proofs are omitted. 
\begin{theorem}
Assume that $k$ is a perfect field, and let $\bar{k}$ be an algebraic 
closure of $k$. Let $C = \mathscr{R}_{k}\, K_{0}$ be the Riemann surface
of an algebraic function field $K_{0}/k$.
The absolut Galois group of $\mathfrak{G}_{\bar{k}/k}$
acts on the closed points of curve $$C_{\bar{k}} = C\times_{{\rm Spec}\, k}{\rm Spec}\,\bar{k}$$
which can be identified with the set of the $\bar{k}$-rational points of $C(\bar{k})$.
The orbit space $C(\bar{k})/\mathfrak{G}_{\bar{k}/k}$ can be identified
with $C$. We denote by $\pi : C(\bar{k})\rightarrow C$ the quotient map.
In terms of the places, $\pi$ is defined as the map 
$$\mathscr{R}_{\bar{k}}\, K_{0}\otimes_{k}\bar{k}\rightarrow \mathscr{R}_{k}\, K_{0},\,\,\, P\mapsto P\cap K_{0}.$$ 
If $\mathfrak{D} = \sum_{z\in C}\Delta_{z}\cdot z$ 
is a proper $\sigma$-polyhedral divisor over $C$, then
\begin{eqnarray*}
A[C,\mathfrak{D}]\otimes_{k}\bar{k} = A\left[C(\bar{k}),
\mathfrak{D}_{\bar{k}}\right], 
\end{eqnarray*}
where $\mathfrak{D}_{\bar{k}}$ is the proper $\sigma$-polyhedral divisor over
$C(\bar{k})$ defined by
\begin{eqnarray*}
\mathfrak{D}_{\bar{k}} = \sum_{z\in C}\Delta_{z}\cdot \pi^{\star}(z)\,\,\,\rm with\,\,\,\it 
\pi^{\star}(z) = \sum_{z'\in \pi^{-\rm 1\it}(z)}z'. 
\end{eqnarray*}
\end{theorem}
\begin{proof}
Given a Weil $\mathbb{Q}$-divisor $D$ over $C$, by [Sti93, Theorem $3.6.3$] we obtain
\begin{eqnarray*}
H^{0}(C,\mathcal{O}_{C}(\lfloor D\rfloor))\otimes_{k}\bar{k} = 
H^{0}(C(\bar{k}),\mathcal{O}_{C(\bar{k})}(\lfloor \pi^{\star}\,D\rfloor)).
\end{eqnarray*}
The proof reduces to a computation of $A[C,\mathfrak{D}]\otimes_{k}\bar{k}$. 
The properness of $\mathfrak{D}_{\bar{k}}$ is given for instance by $3.5\, \rm (i)$.
\end{proof}
\begin{remarque}
It is well known that every finitely generated extension of 
a perfect field is separable. In the non-perfect case,   
we may consider the inseparable algebraic function field
of one variable
\begin{eqnarray*}
K_{0} = \rm Frac\,\it \frac{k[X,Y]}{(tX^{\rm 2\it}+s+Y^{\rm 2\it})}\,,   
\end{eqnarray*}
where  $k = \mathbb{F}_{2}(s,t)$ is the rational function field
in two variables. Consequently,
for any proper polyhedral divisor $\mathfrak{D}$ over $C = \mathscr{R}_{k}\, K_{0}$,
the ring $A[C,\mathfrak{D}]\otimes_{k}\bar{k}$ contains a 
nonzero nilpotent element.    
\end{remarque}

\section{Split affine $\mathbb{T}$-varieties of complexity one}
As an application of the results in the previous sections, we can give now a combinatorial description
of affine $\mathbb{T}$-varieties of complexity one over any field $k$. 

\begin{rappel}
Let $\mathbb{T}$ be a split algebraic torus over $k$.
Denote by $M$ and $N$ its dual lattices of characters and of one-parameter subgroups, respectively. 
Let $X = \rm Spec\,\it A$ be an affine variety over $k$. 
Assume that $\mathbb{T}$ acts on $X$, respectively. Then the associated morphism 
$A\rightarrow A\otimes_{k}k[\mathbb{T}]$ endows $A$ with 
an $M$-grading. Conversely, an $M$-grading on the algebra 
$A$ yields naturally a $\mathbb{T}$-action on $X$. 
We say that $X$ is a \em $\mathbb{T}$-variety \rm if $X$ is normal and if
the $\mathbb{T}$-action on $X$ is effective\footnote{Seeing $\mathbb{T}$ as a representable
group functor, this means that the kernel of the natural
transformation of group functors $\mathbb{T}\rightarrow\rm Aut\,\it X$ is trivial.}. 
This latter is equivalent to the condititon that $A$ is normal and
the set of its weights generates the lattice $M$.
\end{rappel}

\begin{definition}
\begin{enumerate}
\item[(1)] Given an effective $\mathbb{T}$-action on the variety $X$, a \emph{multiplicative system} of $k(X)$ 
is a sequence $(\chi^{m})_{m\in M}$,
where each $\chi^{m}$ is a homogeneous element of $k(X)$ of degree $m$ such
that $\chi^{m}\cdot \chi^{m'} = \chi^{m + m'}$, for all $m,m'\in M$, and $\chi^{0} = 1$.
\item[(2)]
Let $C$ be a regular curve over $k$, and let $\sigma\subset N_{\mathbb{Q}}$
be a strongly convex polyhedral cone. A $\sigma$-polyhedral divisor $\mathfrak{D} = \sum_{z\in C}\Delta_{z}\cdot z$ 
is called \em proper \rm if it
satisfies one of the following conditions.
\begin{enumerate}
\item[\rm (i)]
$C$ is affine. In particular, $\mathfrak{D}$ is a polyhedral divisor over the Dedekind ring
$A_{0} = k[C]$. 
\item[\rm (ii)]
$C$ is projective, and $\mathfrak{D}$ is a proper polyhedral divisor in the sense of $3.4$. 
\end{enumerate}
We denote by $A[C,\mathfrak{D}]$ the associated $M$-graded algebra.
\end{enumerate}
\end{definition}
Combining $2.5$ and $3.5$ one can describe an affine 
$\mathbb{T}$-variety of complexity one over an arbitrary field via a proper polyhedral divisor.
\begin{theorem}
\
\begin{enumerate}
\item[\rm (i)] To any affine $\mathbb{T}$-variety $X = \rm Spec\,\it A$ of complexity one over $k$,
one can associate a pair $(C_{X}, \mathfrak{D}_{X,\gamma})$, where
\begin{enumerate}
\item[(a)] $C_{X}$ is the abstract regular curve over $k$ defined by the conditions 
$k[C_{X}] = k[X]^{\mathbb{T}}$ and $k(C_{X}) = k(X)^{\mathbb{T}}$.
\item[(b)] $\mathfrak{D}_{X,\gamma}$ is a proper $\sigma_{X}$-polyhedral divisor 
over $C_{X}$, which depends only on $X$ and on a multiplicative system $\gamma = (\chi^{m})_{m\in M}$ of $k(X)$.
For each primitive vector $m\in M$ in the relative interior of $\sigma_{X}^{\vee}$, we
denote by $\pi_{m}: V_{m}\rightarrow C_{X}$ the quotient map for the $\mathbb{G}_{m}$-action
on $V_{m}$, where $$V_{m} = {\rm Spec}\, A_{(m)}\setminus \mathbb{V}(\bigoplus_{r\in\mathbb{Z}_{>0}}A_{rm})$$
and $A_{(m)} = \bigoplus_{r\geq 0}A_{rm}$. The polyhedral divisor $\mathfrak{D}_{X,\gamma}$ is given by the relation
$$\mathfrak{D}_{X,\gamma}(m) = \sum_{Z\subset V_{m}}\frac{p_{Z}}{q_{Z}}\,(\pi_{m})_{\star}(\bar{Z}),$$
where ${\rm div}_{V_{m}}(\chi^{m}) = \sum_{Z\subset V_{m}}p_{Z}\cdot Z$, the prime divisor $\bar{Z}$ is the closure of $Z$
in ${\rm Spec}\, A_{(m)}$, and\footnote{Note that every closed subscheme $\bar{Z}\subset {\rm Spec}\, A_{(m)}$ such
that $Z$ is a prime divisor contained in the support of ${\rm div}(\chi^{m})$ is $\mathbb{G}_{m}$-stable.
Hence $I(\bar{Z})$ is a graded ideal of $A_{(m)}$ and the number $q_{Z}$ is well defined.}  
$$\left\{
\begin{array}{l}
  q_{Z} =  {\rm g.c.d.}\{r\in\mathbb{Z}_{>0}\,|\, (A_{(m)}/I(\bar{Z}))_{rm}\neq \{0\}\}\,\,\,{\rm if}\,\,\,p_{Z}\neq 0,\,\,\,{\rm and}\\
  q_{Z} = 1\,\,\,{\rm otherwise.}
\end{array}
\right.$$ 
\end{enumerate}
Moreover, we have a natural identification $A = A[C_{X},\mathfrak{D}_{X,\gamma}]$ of $M$-graded algebras
with the property that every homogeneous element $f\in A$ of degree $m$ is equal to $f_{m}\chi^{m}$ 
for a unique global section $f_{m}$ of the sheaf $\mathcal{O}_{C_{X}}(\lfloor \mathfrak{D}_{X,\gamma}(m)\rfloor)$.

\item[\rm (ii)] Conversely, if $\mathfrak{D}$ is a proper $\sigma$-polyhedral 
divisor on a regular curve $C$ over $k$, then $X = {\rm Spec}\,A[C,\mathfrak{D}]$ defines an 
affine $\mathbb{T}$-variety of complexity one over $k$. 
\end{enumerate}
\end{theorem}
\begin{proof}
$\rm (i)$ Let $\sigma_{X}\subset N_{\mathbb{Q}}$ be the dual of the weight cone of $A$.
Choosing a multiplicative system $\gamma  = (\chi^{m})_{m\in M}$, we have
an embedding
\begin{eqnarray*}
A\subset \bigoplus_{m\in M}k(X)^{\mathbb{T}}\,\chi^{m} = k(X)^{\mathbb{T}}[M].
\end{eqnarray*}
Furthermore, $A$ and $k(X)^{\mathbb{T}}[M]$ have
the same field of fractions. The graded piece $A_{0}$ is the algebra of $\mathbb{T}$-invariants.
Denote by $K_{0}$ the field of fractions of $A_{0}$.
Assume that $A_{0}\neq k$. Then we have $K_{0} = k(X)^{\mathbb{T}}$. 
Indeed, by assumption every algebraic element of $K_{0}$
over $k$ belongs to $k$. Therefore the transcendence degree
of $K_{0}/k$ is equal to $1$ so that $k(X)^{\mathbb{T}}/K_{0}$ is algebraic. 
Using the normality of $A_{0}$ one concludes that $K_{0} = k(X)^{\mathbb{T}}$.
Remark further that the ring $A_{0}$ is a Dedekind domain. 
By Theorem $2.5\,\rm (ii)$, we obtain that $A = A[C_{X},\mathfrak{D}_{X, \gamma}]$ for
a unique $\sigma_{X}$-polyhedral divisor $\mathfrak{D}_{X,\gamma}$ over $A_{0}$.
If $A_{0} = k$, then one concludes by Theorem $3.5\,\rm (i)$. The characterization of
$\mathfrak{D}_{X,\gamma}(m)$ in term of the principal divisor ${\rm div}_{V_{m}}(\chi^{m})$,
for every primitive vector $m\in\sigma_{X}^{\vee}\cap M$, is a consequence of [Dem88, 3.5].  
Assertion $\rm (ii)$ follows immediately from 
$2.5\,\rm (i)$ and $3.5\, \rm (ii)$.
\end{proof}
\begin{rappel}
By a \em principal $\sigma$-polyhedral divisor \em $\mathfrak{F}$ over $C$
we mean a pair $(\varphi, \mathfrak{D})$ with a semigroup morphism  
$\varphi : \sigma^{\vee}_{M}\rightarrow k(C)^{\star}$ and a $\sigma$-polyhedral 
divisor $\mathfrak{D}$ over $C$ such that for any $m\in\sigma^{\vee}_{M}$ we have
\begin{eqnarray*}
\mathfrak{D}(m) = \rm div_{\it C}\,\it \mathfrak{F}(m). 
\end{eqnarray*}
Starting with $\mathfrak{F}$ and choosing a finite generating set of $\sigma^{\vee}_{M}$ 
one can easily construct $\mathfrak{D}$ satisfying the 
equalities as before. Usually we denote $\mathfrak{F}$
and $\mathfrak{D}$ by the same symbol. 
\end{rappel}
The following result provides a description of equivariant isomorphisms between
two affine $\mathbb{T}$-varieties of complexity one over an arbitrary field.
See [AH06, Section $8,9$] for higher complexity when the ground field is algebraically
closed of characteristic zero. 
\begin{proposition}
Let $X_{1}, X_{2}$ be two affine $\mathbb{T}$-varieties of complexity one over $k$
described by the respective pairs $(C_{X_{1}},\mathfrak{D}_{X_{1},\gamma_{1}})$ and
$(C_{X_{2}},\mathfrak{D}_{X_{2},\gamma_{2}})$ {\rm (}see {\rm 4.3}{\rm )}. Then 
the $\mathbb{T}$-varieties $X_{1}$ and $X_{2}$ are $\mathbb{T}$-isomorphic if and only if
there exist an isomorphism $\phi: C_{X_{1}}\rightarrow C_{X_{2}}$ of algebraic curves 
over $k$, an automorphism\footnote{For any morphism $F:N\rightarrow N,$
we denote by $F_{\mathbb{Q}}:N_{\mathbb{Q}}\rightarrow N_{\mathbb{Q}}$ the induced $\mathbb{Q}$-linear map of $F$.}
 of lattices $F:N\rightarrow N$ satisfying $F_{\mathbb{Q}}(\sigma_{X_{1}})
=\sigma_{X_{2}}$, and a principal $\sigma_{X_{2}}$-polyhedral divisor $\mathfrak{F}$ over $C_{X_{1}}$
such that for all $m\in \sigma_{X_{2}}^{\vee}\cap M$,
$$\phi^{\star}(\mathfrak{D}_{X_{2},\gamma_{2}})(m) = F_{\star}(\mathfrak{D}_{X_{1},\gamma_{1}})(m)
+\phi^{\star}(\mathfrak{F})(m),$$
where 
$$\mathfrak{D}_{X_{1},\gamma_{1}} = \sum_{z\in C_{X_{1}}}\Delta_{z}^{1}\cdot z,\,\,\,
\mathfrak{D}_{X_{2},\gamma_{2}} = \sum_{z\in C_{X_{2}}}\Delta_{z}^{2}\cdot z,\,\,\, and$$
$$F_{\star}(\mathfrak{D}_{X_{1},\gamma_{1}}) = \sum_{z\in C_{X_{1}}}(F_{\mathbb{Q}}(\Delta_{z}^{1}) + \sigma_{X_{2}})\cdot z,
\,\,\,\phi^{\star}(\mathfrak{D}_{X_{2},\gamma_{2}}) = \sum_{z\in C_{X_{2}}}\Delta_{z}^{2}\cdot \phi^{-1}(z).$$
\end{proposition}
\begin{proof}
Let $\psi:k[X_{2}]\rightarrow k[X_{1}]$ be an isomorphism of $M$-graded algebras over $k$.
Then $\psi$ induces two natural maps 
\begin{align*}
\phi^{\star}:k(C_{X_{2}}) = k(X_{2})^{\mathbb{T}}\rightarrow k(C_{X_{1}}) = k(X_{1})^{\mathbb{T}},
 \\
\phi^{\star}:k[C_{X_{2}}] = k[X_{2}]^{\mathbb{T}}\rightarrow k[C_{X_{1}}] = k[X_{1}]^{\mathbb{T}},
\end{align*} 
which yield a morphism of algebraic curves $\phi:C_{X_{1}}\rightarrow C_{X_{2}}$. Moreover, we define an automorphism
$F:N\rightarrow N$ such that $F_{\mathbb{Q}}(\sigma_{X_{1}}) = \sigma_{X_{2}}$ and a principal $\sigma_{X_{2}}$-polyhedral
divisor $\mathfrak{F}$ as follows. If $\gamma_{1} = (\chi^{m})_{m\in M}$ and $\gamma_{2} = (\xi^{m})_{m\in M}$, then for every 
homogeneous element $f\xi^{m}\in k[X_{2}]$ there exists a unique vector $F^{\vee}(m)\in\sigma_{X_{1}}^{\vee}\cap M$
such that
\begin{eqnarray}
\psi(f\xi^{m}) = \phi^{\star}(f)\cdot \phi^{\star}(f_{m})\chi^{F^{\vee}(m)}.
\end{eqnarray}
The map $F$ is the dual morphism of $F^{\vee}$. Extending $\psi$ to an isomorphism $\bar{\psi}:k(X_{2})\rightarrow k(X_{1}),$
we can define
the semigroup morphism 
$$\sigma_{X_{2}}^{\vee}\cap M\rightarrow k(C_{X_{2}})^{\star},\,\,\, m\mapsto f_{m}^{-1},$$
which gives the principal polyhedral divisor $\mathfrak{F}$. Now we conclude by the
equivalences
\begin{align*}
{\rm div}(\phi^{\star}(f))  + \phi^{\star}(\mathfrak{D}_{X_{2},\gamma_{2}})(m)\geq 0 
\\
\Leftrightarrow {\rm div}(f) + \mathfrak{D}_{X_{2},\gamma_{2}}(m)\geq 0
\\
\Leftrightarrow {\rm div}(\phi^{\star}(f)\cdot \phi^{\star}(f_{m})) + \mathfrak{D}_{X_{1},\gamma_{1}}(F^{\vee}(m))\geq 0
\\ 
\Leftrightarrow {\rm div}(\phi^{\star}(f)) + F_{\star}(\mathfrak{D}_{X_{1},\gamma_{1}})(m) + \phi^{\star}(\mathfrak{F})(m)\geq 0, 
\end{align*}
which imply the equality $\phi^{\star}(\mathfrak{D}_{X_{2},\gamma_{2}})(m) = F_{\star}(\mathfrak{D}_{X_{1},\gamma_{1}})(m)
+\phi^{\star}(\mathfrak{F})(m)$, for every $m\in\sigma_{X_{2}}^{\vee}\cap M$.
Conversely, starting with the triple $(F,\phi, \mathfrak{F})$, we define a morphism $\psi: k[X_{2}]\rightarrow
k[X_{1}]$ by the equality (1), where $\phi^{\star}:k(C_{X_{2}})\rightarrow k(C_{X_{1}})$ is the comorphism
of $\phi$, $F$ is the dual of $F^{\vee}$, and $\mathfrak{F}$ is given by the morphism $m\mapsto f^{-1}_{m}$. 
A straightforward verification shows that $\psi$ is well defined and gives an isomorphism of $M$-graded algebras.
\end{proof}

\section{Non-split case via Galois descent}
In view of the result in the previous section, we provide a combinatorial description 
of affine normal varieties endowed with a (non-necessary split) torus action of complexity one
(see $5.4$ for a precise definition). This can be compared with well-known descriptions 
for toric and spherical varieties, see [Bry79, CTHS05, Vo82, ELST12, Hur11]. 
\begin{rappel}
For a field extension $F/k$ and a scheme $X$ over $k$
we let 
\begin{eqnarray*}
X_{F} = X\times_{\rm Spec\,\it k}\rm Spec\,\it F.
\end{eqnarray*}
This is a scheme over $F$. 
An algebraic \em torus \rm of dimension $n$
is an algebraic group $\mathbf{G}$ over $k$ such that there exists
a finite Galois extension $E/k$ yielding an isomorphism of
algebraic groups $\mathbf{G}_{E}\simeq \mathbb{G}_{m,E}^{n}\,(\star)$, where
$\mathbb{G}_{m}$ is the multiplicative group scheme over $k$.
We say that the torus $\mathbf{G}$ \em splits in the extension $E/k$ \rm 
if we have an isomorphism similar to $(\star)$. For more details concerning the
theory of non-split reductive group the reader may consult [BT65, Spr98].
\end{rappel}
Let $\mathbf{G}$ be a torus over $k$ that splits in a finite Galois
extension $E/k$. Denote by $\mathfrak{G}_{E/k}$ the Galois
group of $E/k$. Consider also 
the dual lattices $M$ and $N$, respectively, of characters and of one-parameter 
subgroups of the split torus $\mathbf{G}_{E}$. In the sequel most 
of our varieties will be defined over the field $E$.
We start by recalling the following classical notion.
\begin{definition}
\
\begin{enumerate}
\item[\rm (i)]
A $\mathfrak{G}_{E/k}$-action
on a variety $V$ over $E$ is called \em semi-linear \rm if $\mathfrak{G}_{E/k}$ acts by 
scheme automorphisms, and if for any 
$g\in \mathfrak{G}_{E/k}$ the diagram
\begin{eqnarray*}
  \xymatrix{
    V \ar[r]^g \ar[d] & V \ar[d] \\ 
    \rm Spec\,\it E \ar[r]_g & \rm Spec\,\it E}
\end{eqnarray*}
is commutative.
\item[\rm (ii)] Let $B$ be an algebra over $E$. A \em semi-linear \rm $\mathfrak{G}_{E/k}$-action on $B$ 
is an action by ring automorphisms such that for all $a\in B$, $\lambda\in E$,
and $g\in \mathfrak{G}_{E/k}$
\begin{eqnarray*}
g\cdot (\lambda a) = g(\lambda)g\cdot a.
\end{eqnarray*} 
\end{enumerate}
If $V$ is affine, then defining a semi-linear $\mathfrak{G}_{E/k}$-action on $V$ is
equivalent to defining a semi-linear $\mathfrak{G}_{E/k}$-action on the algebra
$E[V]$.
\end{definition}
Next, we recall a well-known description of algebraic tori related to finite groups actions on lattices.
\begin{rappel}
The Galois group $\mathfrak{G}_{E/k}$ acts naturally on the torus
\begin{eqnarray*}
\mathbf{G}_{E} = \mathbf{G}\times_{\rm Spec\,\it k}\rm Spec \,\it E
\end{eqnarray*}
by action on the second factor. The corresponding action on $E[M]$
is determinated by a linear 
$\mathfrak{G}_{E/k}$-action on $M$ (see e.g. [ELST12, Proposition $2.5$], [Vos82, Section $1$])
permuting the Laurent monomials.

Conversely, given a linear 
$\mathfrak{G}_{E/k}$-action on $M$ we have a semi-linear action on 
$E[M]$ defined by
\begin{eqnarray*}
g\cdot (\lambda\chi^{m}) = g(\lambda)\chi^{g\cdot m},
\end{eqnarray*}
where $g\in \mathfrak{G}_{E/k}$, $\lambda\in E$ and $m\in M$, respects 
the Hopf algebra structure. As a consequence of the Speiser's Lemma, we obtain
a torus $\mathbf{G}$ over $k$ that splits in $E/k$. In addition, the semi-linear
action that we have defined on $\mathbf{G}_{E} =  \mathbf{G}\times_{\rm Spec\,\it k}\rm Spec \,\it E$ is exactly
the natural semi-linear action on the second factor. 
\end{rappel} 
Let us further introduce our category of $\mathbf{G}$-varieties.
\begin{definition}  
A \em $\mathbf{G}$-variety of complexity \rm $d$ (splitting
in $E/k$) is a normal variety over $k$ with a $\mathbf{G}$-action
and such that $X_{E}$ is a $\mathbf{G}_{E}$-variety of complexity $d$ in the sense of Section $4$.
A \em $\mathbf{G}$-morphism \rm between $\mathbf{G}$-varieties $X,Y$ over $k$ is a morphism $f:X\rightarrow Y$
of varieties over $k$ such that
\begin{eqnarray*}
	\xymatrix{
    \mathbf{G}\times X \ar[r]^{\rm id\it\times f} \ar[d] & \mathbf{G}\times Y \ar[d] \\ 
    X \ar[r]_f & Y}
\end{eqnarray*}
 is commutative.
\end{definition}
An important class of semi-linear actions is provided by the actions respecting a split torus action.
The $\mathfrak{G}_{E/k}$-action on $\mathbf{G}_{E}$ is defined as in $5.3$.
\begin{definition}
\
\begin{enumerate}
\item[\rm (i)] Let $B$ be an $M$-graded algebra over $E$. A semi-linear $\mathfrak{G}_{E/k}$-action
on $B$ is called \em homogeneous \rm if it sends homogeneous elements into homogeneous elements. 
\item[\rm (ii)] A semi-linear $\mathfrak{G}_{E/k}$-action on a $\mathbf{G}_{E}$-variety $V$ \em respects 
the $\mathbf{G}_{E}$-action \rm 
if the following diagram
\begin{eqnarray*}
		\xymatrix{
    \mathbf{G}_{E}\times V \ar[r]^{g\times g} \ar[d] & \mathbf{G}_{E}\times V \ar[d] \\ 
    V \ar[r]_g & V}
\end{eqnarray*}
commutes, where $g$ runs $\mathfrak{G}_{E/k}$. 
\end{enumerate}
Assuming that $V$ is affine, a semi-linear $\mathfrak{G}_{E/k}$-action on the variety $V$ respecting the
$\mathbf{G}_{E}$-action corresponds to a homogeneous semi-linear $\mathfrak{G}_{E/k}$-action 
on the algebra $E[V]$.
\end{definition}
The following result is classically stated for the category of quasi-projective varieties
(see the proof of [Hur12, $1.10$]).
In the setting of affine $\mathbf{G}$-varieties we include a short argument.
\begin{lemme}
Let $V$ be an affine $\mathbf{G}_{E}$-variety of complexity $d$ over $E$ 
with a semi-linear $\mathfrak{G}_{E/k}$-action. 
Then the quotient $X = V/\mathfrak{G}_{E/k}$ is an affine $\mathbf{G}$-
variety of complexity $d$. We have a natural isomorphism of $\mathbf{G}_{E}$-varieties
$X_{E}\simeq V$ respecting the $\mathfrak{G}_{E/k}$-actions.
\end{lemme}
\begin{proof}
It is known that $R = B^{\mathfrak{G}_{E/k}}$ is finitely generated.
Let us show that $R$ is normal. Letting $L$ be the field of fractions of $R$ and considering an element $f\in L$ 
integral over $R$, by the normality of $B$, we have $f\in B\cap L = R$. This proves that $R$ is normal.
Using the above definition, the variety $X$ is endowed with a $\mathbf{G}$-action. The rest of the proof
follows from Speiser's Lemma. 
\end{proof}
Fixing an affine $\mathbf{G}$-variety $X$ of complexity $d$ over $E$, an \em $E/k$-form \rm of $X$
is an affine $\mathbf{G}$-variety $Y$ over $k$ such that we have a $\mathbf{G}_{E}$-isomorphism
$X_{E}\simeq Y_{E}$. Our aim is to give a combinatorial description of $E/k$-forms of $X$.
Let us recall first in this context some notion of non-abelian Galois cohomology (see e.g. [Ser94, III Section $1$]
for the category of varieties). 
\begin{rappel}
Let $Y, Y'$ be $E/k$-forms of the fixed affine $\mathbf{G}$-variety $X$. The Galois group 
$\mathfrak{G}_{E/k}$ acts on the set of $\mathbf{G}_{E}$-isomorphisms
between $Y_{E}$ and $Y'_{E}$. Consequently, it acts also by group automorphisms on the
group of $\mathbf{G}_{E}$-automorphisms $\rm Aut_{\it\mathbf{G}_{E}}(\it X_{E}\rm )$ of $X_{E}$.
More precisely, recall that for any $g\in\mathfrak{G}_{E/k}$ and 
any $\mathbf{G}_{E}$-isomorphism $\varphi: Y_{E}\rightarrow Y'_{E}$
one defines $g(\varphi)$ by the following commutative diagram
\begin{eqnarray*}
   \xymatrix{
    Y_{E} \ar[r]^{g(\varphi)} \ar[d]_{g} & Y'_{E} \ar[d]^{g} \\ 
    Y_{E} \ar[r]_{\varphi} & Y'_{E}}.
\end{eqnarray*}
Note that this $\mathfrak{G}_{E/k}$-action depends on the choice of
the $E/k$-forms $Y,Y'$. Now given a $\mathbf{G}_{E}$-isomorphism 
$\psi: X_{E}\rightarrow Y_{E}$ the map 
\begin{eqnarray*}
a:\mathfrak{G}_{E/k}\rightarrow \rm Aut_{\it\mathbf{G}_{E}}(\it X_{E}\rm ),\it
\,\,\, g\mapsto a_{g} = \psi^{\rm -1\it}\circ g(\psi)
\end{eqnarray*}
is a \em $1$-cocycle. \rm This means that for all $g,g'\in \mathfrak{G}_{E/k}$ we have
\begin{eqnarray*}
a_{g}\circ g(a_{g'}) = \psi^{-1}\circ g(\psi)\circ g\left(\psi^{-1}\circ g'(\psi)\right) = a_{gg'}.
\end{eqnarray*}
Let $\phi :Y\rightarrow Y'$ be a $\mathbf{G}$-isomorphism and consider a $\mathbf{G}_{E}$-isomorphism
$\varphi : X_{E}\rightarrow Y'_{E}$ yielding a $1$-cocycle $b$ as above. The diagram
\begin{eqnarray*}
   \xymatrix{
    X_{E} \ar[r]^{\psi} \ar[d]_{\alpha} & Y_{E} \ar[d]^{\phi' = \phi\times\rm id} \\ 
    X_{E} \ar[r]_{\varphi} & Y'_{E}}
\end{eqnarray*}
is commutative, where $\alpha\in \rm Aut_{\it\mathbf{G}_{E}}(\it X_{E}\rm )$ and $\phi'$ is the extension $\phi$. 
Since for any $g\in \mathfrak{G}_{E/k}$ we have 
$g(\phi') = \phi'$, it follows that
\begin{eqnarray*}
b_{g} =  \alpha\circ a_{g}\circ g\left(\alpha^{-1}\right).
\end{eqnarray*} 
In this case we say that the cocycles $a$ and $b$ are \em cohomologous. \rm
We obtain as well a map $\Phi$ between the pointed set of 
isomorphism classes of $E/k$-forms of $X$ and the pointed 
set 
\begin{eqnarray*}
H^{1}(E/k, \rm Aut_{\it\mathbf{G}_{E}}(\it X_{E}\rm ))
\end{eqnarray*}
of cohomology classes of $1$-cocycles $a : \mathfrak{G}_{E/k}
\rightarrow \rm Aut_{\it\mathbf{G}_{E}}(\it X_{E}\rm )$.

Conversely, starting with a cocycle $a$ the map
\begin{eqnarray*}
\mathfrak{G}_{E/k} \rightarrow \rm Aut_{\it\mathbf{G}_{E}}(\it X_{E}\rm ),\,\,\,
\it g\mapsto a_{g}\circ g
\end{eqnarray*}
is a semi-linear action on $X_{E}$ respecting the $\mathbf{G}_{E}$-action. 
According to Lemma $5.6$ one can obtain an 
$E/k$-form $W$ of $X$ by taking
the quotient $X_{E}/\mathfrak{G}_{E/k}$. Changing $a$ by a cohomologous 
$1$-cocycle gives an $E/k$-form of $X$ isomorphic to $W$. Thus we
deduce that the map $\Phi$ is bijective. 

Moreover, let $\gamma$ be a semi-linear
$\mathfrak{G}_{E/k}$-action on $X_{E}$. Remark that
\begin{eqnarray*}
   \xymatrix{
    X_{E} \ar[r]^{\gamma(g')} \ar[d]_{g^{-1}} & X_{E}\ar[d]^{g^{-1}} \\ 
    X_{E} \ar[r]_{g(\gamma(g'))} & X_{E}}
\end{eqnarray*}
commutes for all $g,g'\in \mathfrak{G}_{E/k}$. Hence the 
equality $a_{g} = \gamma(g)\circ g^{-1}$
defines a $1$-cocycle $a$. A straightforward verification shows that 
$H^{1}(E/k, \rm Aut_{\it\mathbf{G}_{E}}(\it X_{E}\rm ))$ is also in bijection
with the pointed set of conjugacy classes of semi-linear $\mathfrak{G}_{E/k}$-actions
on $X_{E}$ respecting the $\mathbf{G}_{E}$-action.
\end{rappel}
As explained in the above paragraph, classifying the pointed set of $E/k$-forms of $X$
is equivalent to classifying all possible semi-linear $\mathfrak{G}_{E/k}$-actions
on $X_{E}$. Thus generalizing the notion of proper polyhedral divisors, we consider the combinatorial
counterpart of this classification.
\begin{definition}
Let $C$ be a regular curve over $E$ and let $\sigma\subset N_{\mathbb{Q}}$ be a 
strongly convex cone.
A \em $\mathfrak{G}_{E/k}$-invariant
$\sigma$-polyhedral divisor \rm over $C$ is a $4$-tuple $(\mathfrak{D},\mathfrak{F},\star,\cdot)$
verifying the following conditions.
\begin{enumerate}
\item[\rm (i)] $\mathfrak{D}$ is a proper 
$\sigma$-polyhedral divisor over $C$.
\item[\rm (ii)] The curve $C$ is endowed with a semi-linear $\mathfrak{G}_{E/k}$-action
\begin{eqnarray*}
\mathfrak{G}_{E/k}\times C\rightarrow C, \,\,\, (g,z)\mapsto g\star z.
\end{eqnarray*}
This yields naturally an action on the space of Weil $\mathbb{Q}$-divisors
over $C$. More precisely, given $g\in\mathfrak{G}_{E/k}$ and a $\mathbb{Q}$-divisor $D$
over $C$ we let 
\begin{eqnarray*}
g\star D = \sum_{z\in C}a_{g^{-1}\star z}\cdot z,\,\,\,\rm where\,\,\,\it D = \sum_{z\in C}a_{z}\cdot z.
\end{eqnarray*}
\item[(iii)] The lattice $M$ is endowed with a linear $\mathfrak{G}_{E/k}$-action
\begin{eqnarray*}
\mathfrak{G}_{E/k}\times M\rightarrow M,\,\,\, (g,m)\mapsto g\cdot m
\end{eqnarray*}
preserving the subset $\sigma^{\vee}_{M}$. 
\item[(iv)]Moreover, there is a map $g\mapsto f_{g},\,\, \mathfrak{G}_{E/k}\rightarrow {\rm Hom}(M, E(C)^{\star})$ satisfying $$f_{gh}(m) = g(f_{h}(m))f_{g}(h\cdot m)$$
for all $g, h\in \mathfrak{G}_{E/k}$ and $m\in M$. In addition, we ask that
$$g\star(\mathfrak{D}(m)) = \mathfrak{F}_{g}(m) + \mathfrak{D}(g\cdot m)\text{ for all }g\in \mathfrak{G}_{E/k}\text{ and } m\in\sigma^{\vee}_{M}, $$
$$\text{ where } \mathfrak{F}_{g}(m) := {\rm div}(f_{g}(m)).$$
We see $\mathfrak{F}_{g}$ as a principal polyhedral divisor and $\mathfrak{F}$ denotes the map $g\mapsto \mathfrak{F}_{g}$.
\end{enumerate}
\end{definition}
The following result allows to simplify the Galois invariant polyhedral divisor description in a particular case.
\begin{lemme}
Let $E_{0}/K_{0}$ be a finite Galois extension with Galois group $\mathfrak{G}_{E_{0}/K_{0}}$.
Assume that $\mathfrak{G}_{E_{0}/K_{0}}$ acts linearly on $M$.
For any $g\in \mathfrak{G}_{E_{0}/K_{0}}$ consider a morphism of groups 
$f_{g}: M\rightarrow E_{0}^{\star}$ satisfying the equalities
\begin{eqnarray*}
f_{gh}(m) = g\left(f_{h}(m)\right)f_{g}(h\cdot m),
\end{eqnarray*}
where $g,h\in\mathfrak{G}_{E_{0}/K_{0}}$ and $m\in M$. Suppose that the torus $T/\mathfrak{G}_{E_{0}/K_{0}}$ is quasi-split, where  $T$ is the torus ${\rm Hom}(M, E_{0}^{\star})$ equipped with the natural $\mathfrak{G}_{E_{0}/K_{0}}$-action $($i.e., we ask that $H^{1}(E_{0}/K_{0}, T) = 1)$.
Then there exists a morphism of groups $b : M\rightarrow E_{0}^{\star}$
such that for all $g\in\mathfrak{G}_{E_{0}/K_{0}}$, $m\in M$ we have
\begin{eqnarray*}
f_{g}(m) = b(g\cdot m)g(b(m))^{-1}\,.
\end{eqnarray*}
\end{lemme}
\begin{proof}
The opposite of $\mathfrak{G}_{E_{0}/K_{0}}$ is the group $H$ 
with underlying set $\mathfrak{G}_{E_{0}/K_{0}}$ and the
multiplication law defined by $g\star h = hg$,
where $g,h\in H$. For $g\in H$ we denote
by $a_{g} : M\rightarrow E_{0}^{\star}$ the morphism of groups defined by 
\begin{eqnarray*}
a_{g}(m) = g^{-1}(f_{g}(m)),
\end{eqnarray*} 
where $m\in M$. We can also define an $H$-action by group automorphisms
on the abelian group
\begin{eqnarray*}
T = \rm Hom(\it M,E_{\rm 0}^{\star})
\end{eqnarray*}
over $E_{0}$ by letting $(g\cdot\alpha)(m) = g^{-1}(\alpha(g\cdot m))$,
where $\alpha\in T$, $g\in H$, and $m\in M$. Considering
$g,h\in H$ we obtain
\begin{eqnarray*}
a_{h\star g}(m) = (gh)^{-1}(f_{gh}(m)) = (gh)^{-1}(g(f_{h}(m))f_{g}(h\cdot m))
 = a_{h}(m)(h\cdot a_{g})(m)
\end{eqnarray*}
so that $g\mapsto a_{g}$ is a $1$-cocycle.
One has 
\begin{eqnarray*}
H^{1}(H,T) \simeq H^{1}(E_{0}/K_{0}, T)  = 1.
\end{eqnarray*}
Hence there exists $b\in T$ such that for any $g\in H$ 
we have $a_{g} = b \cdot (g\cdot b^{-1})$. The latter equalities 
imply the result.
\end{proof}
The next theorem yields a classification of affine $\mathbf{G}$-varieties of complexity one in
terms of invariant polyhedral divisors. 
\begin{theorem}
Let $\mathbf{G}$ be a torus over $k$ splitting in a finite Galois extension $E/k$.
Denote by $\mathfrak{G}_{E/k}$ the Galois group of $E/k$.
\begin{enumerate}
\item[\rm (i)] Every affine $\mathbf{G}$-variety of complexity one 
splitting in $E/k$ is described by a $\mathfrak{G}_{E/k}$-invariant proper 
polyhedral divisor over a regular curve. 
\item[\rm (ii)] Conversely, let $C$ be a regular curve over $E$. For a $\mathfrak{G}_{E/k}$-invariant 
proper $\sigma$-polyhedral divisor $(\mathfrak{D},\mathfrak{F},\star,\cdot)$ over $C$ one can endow
the algebra $A[C,\mathfrak{D}]$ with a homogeneous semi-linear $\mathfrak{G}_{E/k}$-action and associate
an affine $\mathbf{G}$-variety of complexity one over $k$ splitting in $E/k$
by letting $X = \rm Spec\,\it A$, where
\begin{eqnarray*}
A = A[C,\mathfrak{D}]^{\mathfrak{G}_{E/k}}.
\end{eqnarray*}
 
\end{enumerate}
\end{theorem}
\begin{proof}
$\rm (i)$ Let $X$ be an affine $\mathbf{G}$-variety of complexity one over $k$.
According to Theorem $4.3$ we may suppose that $B = A[C,\mathfrak{D}]$ is the coordinate ring
of $X_{E}$ for some proper $\sigma$-polyhedral divisor $\mathfrak{D}$ over a regular curve $C$. The algebra
$B$ is endowed with a homogeneous semi-linear $\mathfrak{G}_{E/k}$-action. 
Let $E_{0} = E(C)$. Extending this action on $E_{0}[M]$ we notice that $E_{0}$ and $E[C]$ are 
preserved. Thus we obtain a semi-linear $\mathfrak{G}_{E/k}$-action
on $C$. If $C$ is projective, then one defines the $\mathfrak{G}_{E/k}$-action 
on $C$ in the following way ; given
a place $P\subset E_{0}$ we let 
\begin{eqnarray*}
g\star P = \{g\star f\,|\, f\in P\}.
\end{eqnarray*}
In the case where $C$ is arbitrary, the Speiser Lemma gives the equality
\begin{eqnarray*}
E_{0} = E\cdot K_{0},\,\,\, \rm where\,\,\,\it  K_{\rm 0} = \it E_{\rm 0\it }^{\it \mathfrak{G}_{E/k}}.
\end{eqnarray*}
The finite extension $E_{0}/K_{0}$ is Galois. We have a natural identification 
$\mathfrak{G}_{E/k}\simeq\mathfrak{G}_{E_{0}/K_{0}}$
with the Galois group of $E_{0}/K_{0}$. For all $m\in M$, $g\in\mathfrak{G}_{E/k}$
we have 
\begin{eqnarray}
g\cdot \left(f\chi^{m}\right) = g(f)f_{g}(m)\chi^{\Gamma(g,m)}
\end{eqnarray}
for some some element $f_{g}$ of the abelian group 
$T = \rm Hom(\it M,E_{\rm 0}^{\star})$ and some $\Gamma(g,m)\in M$.
We observe that $\Gamma$ is a linear action on $M$. Denote by $g\cdot m$ the 
lattice vector $\Gamma(g,m)$.
For all $g,h\in\mathfrak{G}_{E/k}$ we have
\begin{eqnarray*}
f_{gh}(m)\chi^{m} = gh\cdot\chi^{m} = g\cdot(h\cdot\chi^{m}) = g(f_{h}(m))f_{g}(h\cdot m)\chi^{gh\cdot m}.
\end{eqnarray*}
First of all, we remark that if $f\in E_{0}^{\star}$ and $g\in\mathfrak{G}_{E/k}$,
then $g\star\rm div\,\it f = \rm div\,\it g(f)$. Let $f\chi^{m}\in B$ be homogeneous of degree $m$. 
The transformation of $f\chi^{m}$ by $g$ is an element of $B$ of degree $g\cdot m$ and so
\begin{eqnarray*}
\rm div\it\, g(f)f_{g}(m) + \mathfrak{D}(g\cdot m)\rm \geq 0.
\end{eqnarray*}
This implies that
\begin{eqnarray*}
g\star\left(-\rm div\,\it f \right)\rm \leq \it
\mathfrak{F}_{g}(m) + \mathfrak{D}(g\cdot m).
\end{eqnarray*}
According to Lemma $1.7$ and Corollary $3.8\,\rm (iii)$ we obtain
\begin{eqnarray*}
 g\star \mathfrak{D}(m)\leq  \mathfrak{F}_{g}(m) + \mathfrak{D}(g\cdot m).
\end{eqnarray*}
The converse inequality uses a similar argument. One concludes that $(\mathfrak{D},\mathfrak{F},\star, \cdot)$
is an invariant polyhedral divisor.

$\rm (ii)$ One defines a 
homogeneous semi-linear $\mathfrak{G}_{E/k}$-action on $A[C,\mathfrak{D}]$
by the equality $(2)$. 
The rest of the proof is a consequence of Lemma $5.6$.
\end{proof}
Let us provide the following elementary example.
\begin{exemple}
Consider the $\sigma$-polyhedral divisor $\mathfrak{D}$ over $\mathbb{A}^{1}_{\mathbb{C}} = \rm Spec\,\it\mathbb{C}[t]$
defined by
\begin{eqnarray*}
((1,0) + \sigma)\cdot \zeta + ((0,1) + \sigma)\cdot (-\zeta) + ((1,-1)+\sigma)\cdot 0,
\end{eqnarray*}
where $\sigma$ is the first quadrant $\mathbb{Q}_{\geq 0}^{2}$ and $\zeta = \sqrt{-1}$.
We endow $\mathfrak{D}$ with a structure
of $\mathfrak{G}_{\mathbb{C}/\mathbb{R}}$-invariant polyhedral divisors by considering $\mathfrak{F}$
induced by the morphism $(m_{1},m_{2})\mapsto t^{m_{2}-m_{1}}$. We have a $\mathfrak{G}_{\mathbb{C}/\mathbb{R}}$-action
\begin{eqnarray*}
\mathfrak{G}_{\mathbb{C}/\mathbb{R}}\rightarrow \rm GL_{2}(\it\mathbb{Z}),\,\,\, g\mapsto \begin{pmatrix}
   0 & 1 \\
   1 & 0 
\end{pmatrix}
\end{eqnarray*}
on the lattice $\mathbb{Z}^{2}$, 
where $g$ is the generator of $\mathfrak{G}_{\mathbb{C}/\mathbb{R}}$. The algebra $\mathbb{C}[t]$ has
the natural complex conjugacy action $\star$ of $\mathfrak{G}_{\mathbb{C}/\mathbb{R}}$. A direct computation
shows that
\begin{eqnarray*}
A = \mathbb{C}\left[t,\frac{1}{t(t-\zeta)}\chi^{(1,0)},\frac{t}{t+\zeta}\chi^{(0,1)}\right]
\end{eqnarray*}
and so $X = \rm Spec \,\it A$ is the affine space $\mathbb{A}^{3}_{\mathbb{C}}$. More concretely,
the $\mathfrak{G}_{\mathbb{C}/\mathbb{R}}$-action on the algebra $A$ is obtained by
\begin{eqnarray*}
g\cdot (f(t)\chi^{(m_{1},m_{2})}) = \bar{f(t)}t^{2(m_{1}-m_{2})}\chi^{(m_{2},m_{1})}. 
\end{eqnarray*}
Letting $x = t^{-1}(1-\zeta)^{-1}\chi^{(1,0)}$ and $y = t(1+\zeta)^{-1}\chi^{(0,1)}$ we obtain
that $A^{\mathfrak{G}_{\mathbb{C}/\mathbb{R}}} = \mathbb{R}[t, x+y, \zeta(x-y)]$.
Hence $X/\mathfrak{G}_{\mathbb{C}/\mathbb{R}}\simeq \mathbb{A}^{3}_{\mathbb{R}}$.
\end{exemple}
Next we describe the pointed set of $E/k$-forms of an affine
$\mathbf{G}$-varieties of complexity one in terms of polyhedral divisors.
\begin{definition}
The invariant $\sigma$-polyhedral divisors $(\mathfrak{D},\mathfrak{F},\star,\cdot)$ and 
$(\mathfrak{D},\mathfrak{F}',\star',\cdot')$ over $C$ are \em conjugated \rm if they verify the following:
there exist $\varphi\in \rm Aut(\it C)$, a principal $\sigma$-polyhedral divisor
$\mathfrak{E}$ over $C$, and a linear automorphism $F\in\rm Aut(\it M)$
giving an automorphism of the $E$-algebra $A[C,\mathfrak{D}]$ (see $4.5$) 
such that for any $g\in\mathfrak{G}_{E/k}$ the diagrams 
\begin{eqnarray*}
  \xymatrix{
    C \ar[r]^{g\star} \ar[d]_{\varphi} & C \ar[d]^{\varphi} \\ 
     C \ar[r]_{g\star'} & C} 
\rm\,\,\, and\,\,\,\it
  \xymatrix{
    M \ar[r]^{g\cdot} \ar[d]_{F} & M \ar[d]^{F} \\ 
     M \ar[r]_{g\cdot'} & M}
\end{eqnarray*}
commute and for any $m\in M$ we have
\begin{eqnarray*}
\mathfrak{E}(g\star m)\cdot \varphi^{\star}(f_{g}(m)) = g(\mathfrak{E}(m))\cdot f_{g}'(F(m)).
\end{eqnarray*}
Consider an affine $\mathbf{G}$-variety $X$ of complexity one described by the invariant polyhedral
divisor $(\mathfrak{D},\mathfrak{F},\star,\cdot)$. 
We denote by $\mathscr{E}_{X}(E/k)$ the pointed set of conjugacy classes of $\mathfrak{G}_{E/k}$-invariant
$\sigma$-polyhedral divisors over $C$ of the form $(\mathfrak{D},\mathfrak{F}',\star',\cdot')$.
\end{definition}
As a direct consequence of the discussion of $5.7$ we obtain the following. 
\begin{corollaire}
Let $C$ be a regular curve over $E$. 
Given an affine $\mathbf{G}$-variety $X$ of complexity one
associated to a $\mathfrak{G}_{E/k}$-invariant polyhedral
divisor $(\mathfrak{D},\mathfrak{F},\star,\cdot)$ over $C$, we have 
a bijection of pointed sets
\begin{eqnarray*}
\mathscr{E}_{X}(E/k)\simeq H^{1}(E/k, \rm Aut_{\it\mathbf{G}_{E}}(\it X_{E}\rm )).
\end{eqnarray*}
\end{corollaire}

\end{document}